\documentclass[10pt,a4paper,oneside,final,reqno]{amsart}

\usepackage[backend=bibtex, style=numeric, maxnames=50]{biblatex}
\renewbibmacro{in:}{}

\DeclareFontFamily{OT1}{pzc}{}
\DeclareFontShape{OT1}{pzc}{m}{it}{<-> s * [1.10] pzcmi7t}{}
\DeclareMathAlphabet{\mathpzc}{OT1}{pzc}{m}{it}

\allowdisplaybreaks

\usepackage{comment}
\usepackage[T1]{fontenc}
\usepackage[utf8]{inputenc}
\usepackage[dvipsnames]{xcolor}
\usepackage[english]{babel}
\usepackage{mathtools, amsthm, amssymb, amscd, amsmath} % Maths
\usepackage{newtxtext, newtxmath} % Different fonts
\usepackage[DIV=11]{typearea} % DIV adjust margin size, DIV=12 is default(?)
\usepackage{hyperref} % Clickable reference (without the ugly border)
\usepackage[notref,notcite]{showkeys} % Showing tex-reference labels in the PDF (when in draft mode, only use in the arXiv version)
\usepackage{todonotes}
\usepackage{enumitem} % Customizing list labels
\usepackage{xypic} % commutative diagrams
\usepackage{tikz}
\usepackage[hmargin=2.5cm,vmargin=1.5cm]{geometry}
\usetikzlibrary{cd, decorations.markings, arrows} % Decorating arrows in graphs

\usepackage[capitalise]{cleveref}

\DeclareMathAlphabet{\mathcal}{OMS}{cmsy}{m}{n} % Reverts \mathcal to ``standard'' (to make O_2 look as it should...) 

\definecolor{DarkPurple}{rgb}{0.40,0.0,0.20}
\hypersetup{
	colorlinks =true,
	linkcolor= DarkPurple, %Plum
	citecolor = NavyBlue %MidnightBlue
} % MidnightBlue, NavyBlue

\usetikzlibrary{cd, decorations.markings, arrows, matrix, calc}
\newcommand{\G}{\mathcal{G}}

\newcommand{\A}{\mathcal{A}}
\newcommand{\Sab}{S_{\lf,(\alpha, \beta)}^{p}}
\newcommand{\Fp}{F_{\lambda}^{p}}

\newcommand{\Cred}{C_{r}^{*}}
\newcommand{\lf}{\mathscr{l}}

\newcommand{\C}{\mathbb{C}}
\newcommand{\R}{\mathbb{R}}

\newcommand{\N}{\mathbb{N}}

\DeclareMathOperator{\ceil}{ceil}

\newtheorem{lemma}{Lemma}[section]
\newtheorem{corollary}[lemma]{Corollary}
\newtheorem{theorem}[lemma]{Theorem}
\newtheorem*{theorem*}{Theorem}
\newtheorem{proposition}[lemma]{Proposition}
\newtheorem{introtheorem}{Theorem}

\theoremstyle{definition}
\newtheorem{definition}[lemma]{Definition}
\newtheorem{example}[lemma]{Example}
\newtheorem{remark}[lemma]{Remark}

\title[]{K-theory invariance of $L^p$-operator algebras associated with étale groupoids of strong subexponential growth}

\author[1]{Are Austad}
\author[2]{Eduard Ortega}
\author[3]{Mathias Palmstrøm}

\address{}
\email{}

\thanks{The first named author was  supported by The Research Council of Norway project 324944. }

\address{Department of Mathematics, Faculty of Mathematics and Natural Sciences, University of Oslo, Oslo, Norway}
\email{areaus@math.uio.no}

\address{Department of Mathematical Sciences, Faculty of Information Technology and Electrical Engineering, NTNU -- Norwegian University of Science and Technology, Trondheim, Norway}
\email{eduard.ortega@ntnu.no}

\address{Department of Mathematical Sciences, Faculty of Information Technology and Electrical Engineering, NTNU -- Norwegian University of Science and Technology, Trondheim, Norway}
\email{mathias.palmstrom@ntnu.no}

\numberwithin{equation}{section} 

\begin{document}
	
\renewcommand{\thefootnote}{}
	
\begin{abstract}
	We introduce 
	the notion of (strong) subexponential growth for étale groupoids and study its basic properties. 
	In particular, we show that the K-groups of the associated groupoid $L^p$-operator algebras are independent of $p \in [1,\infty)$ whenever the groupoid has strong subexponential growth. 
	Several examples are discussed. Most significantly, we apply classical tools from analytic number theory to exhibit an example of an étale groupoid associated with a shift of infinite type which has strong subexponential growth, but not polynomial. 
\end{abstract}
	
\maketitle

\section{Introduction} \label{sec: introduction}
For $p \in [1, \infty)$,  an $L^p$-operator algebra is a Banach algebra which admits an isometric representation on some $L^p$-space. Such Banach algebras were first considered by Herz in \cite{Herz:TheTheoryOfPSpacesWithAnApplicationToConvolution} where he studied the $L^p$-operator algebra generated by the left regular representation of a locally compact group. In the 2010s, Phillips initiated their study once more, leading a program aimed at generalizing the modern theory of C*-algebras to $L^p$-operator algebras, and they have since seen a growing amount of interest among operator algebraists, see for example \cite{Choi15directlyfinite, ChoiGardellaThiel:RigidityResultsForLpOperatorAlgebrasAndApplications, ChungLi2018roerigidity, CortinasRodriguez19OrientedGraphs, Gardella:AModernLook, GardellaLupini:RepresentationsOfEtaleGroupoidsOnLpSpaces, GardellaThiel:GroupAlgebrasActingOnLpSpaces, GardellaThiel:RepresentationsOfPConvolutionAlgebrasOnLqSpaces, PhillipsViola2020}. Both historically and currently, $L^p$-operator algebras are studied mostly through examples, but there have been attempts to establish a more general theory (see \cite{BlecherPhillips:LpOperatorAlgebrasWithApproximateIdentities, GardellaThiel:BanachAlgebrasGeneratedByAnInvertibleIsometryOfAnLpSpace, GardellaThiel:ExtendingRepresentationsOfBanachAlgebrasToTheirBiduals}). However, as of yet, there is no abstract characterization of such Banach algebras like there is for C*-algebras, and lacking the richness of C*-theory, their study often require different techniques from that of C*-algebras.

Given a combinatorial or dynamical object, one can in many cases associate an $L^p$-operator algebra in a way that reflects the combinatorial or dynamical structure of the object, which is a source of interesting examples. Such an object can for example be a graph, a 
group acting on a locally compact space, or a locally compact Hausdorff groupoid. Given such an object and associated $L^p$-operator algebras, for $p \in [1, \infty)$, it is a natural problem to investigate the extent to which the structure of the Banach algebras differ for various exponents $p \in [1, \infty)$. For instance,  natural questions are whether the properties of simplicity or monotraciality are shared by some or all $p \in [1, \infty)$, variants of which have been 
studied in \cite{BaradynKwasniewski:TopFreeActionsAndIdealsInTwistedBanachAlgebraCrossedProducts, BardadynKwaśniewskiMcKee:BanachAlgebrasAssociatedToTwistedÉtaleGroupoidsSimplicityAndPureInfiniteness, PooyaShirin:SimpleReducedLpOperatorCrossedProductsWithUniqueTrace, PhillipsSimplicity2019, Phillips:SimplicityOfUHFAndCuntzAlgebrasOnLpSpaces}.
One crucial aspect in the theory of $L^p$-operator algebras where the value of the exponent plays a significant role is in the so-called rigidity problem
for étale groupoids. In the case where $p=2$, there exist non-isomorphic étale groupoids $\mathcal{G}$ and $\mathcal{G}^{\prime}$ (which are not even continuously orbit equivalent), yet their reduced $C^*$-algebras are isomorphic, that is $C^*_r(\mathcal{G}) \cong C^*_r(\mathcal{G}^{\prime})$.
However, this phenomenon does not occur for the reduced $L^p$-operator algebras for $p\neq 2$ (see \cite{ChoiGardellaThiel:RigidityResultsForLpOperatorAlgebrasAndApplications,GardellaThiel:RigidityResultsGroupsForLpOperatorAlgebras, HetlandOrtega:RigidityOfTwistedGroupoidLpOperatorAlgebras}). This is due to the rigidity of the homotopy classes of invertible isometries on 
$L^p$-spaces, as established by Lamperti's theorem.

Another natural question in this spirit is whether or not their K-theory differs. Phillips computed the K-groups of the $L^p$-Cuntz algebras in \cite{Phillips:CrossedProductsOfLpAlgebrasAndKtheoryOfCuntzAlgebras}, finding that the K-groups are invariant of the exponent $p \in [1,\infty)$. Further, 
Liao and Yu obtained a similar result in \cite{LiaoYu:KTheoryOfGroupBanachAlgebrasAndRD}, to the effect that for a fairly large class of groups, the K-groups of the reduced group $L^p$-operator algebras are invariant of the exponent $p \in [1, \infty)$, and in \cite{WangWang:NoteOnTheEllpToeplitzAlgebra} Wang and Wang showed similarly that the K-groups of the $L^p$-Toeplitz algebras are invariant of the exponent $p \in (1, \infty)$. Inspired by the observation that all of these $L^p$-operator algebras have groupoid models, the authors attempted in \cite{AustadOrtegaPalmstrom:PolynomialGrowthAndPropertyRDpForEtaleGroupoids} to provide a unifying result for all of the aforementioned cases. The attempt was partially successful in that it was proved that for an étale groupoid which has polynomial growth with respect to a continuous length function, the K-groups are indeed invariant of the exponent, see \cite[Theorem 4.7]{AustadOrtegaPalmstrom:PolynomialGrowthAndPropertyRDpForEtaleGroupoids}. This result includes that of Wang and Wang as a special case, as well as Liao and Yu's result when restricting to discrete groups of polynomial growth. However, it fails to include Phillips' result for the $L^p$-Cuntz algebras as their groupoid models have exponential growth (see \cref{ex: a class of examples of exponential growth groupoids}). The purpose of the present paper is to extend \cite[Theorem 4.7]{AustadOrtegaPalmstrom:PolynomialGrowthAndPropertyRDpForEtaleGroupoids} to cover a larger class of étale groupoids, namely those of strong subexponential growth (see \cref{def: (strongly) subexponential growth}). Our first main result is the following. Here $F^p_\lambda (\G)$ is the reduced $L^p$-operator algebra of $\G$ (see \cref{sec: preliminaries}).

\begin{introtheorem} [cf. \cref{thm: K-groups are independent of p}] \label{introthm: strong subexponential growth gives invariance in K-theory}
	If $\G$ is an étale groupoid which has strong subexponential growth with respect to a locally bounded length function, then $K_\ast (F_{\lambda}^{p}(\G))$, for $\ast = 0,1$, is independent of $p \in [1, \infty)$.
\end{introtheorem}

The above result extends \cite[Theorem 4.7]{AustadOrtegaPalmstrom:PolynomialGrowthAndPropertyRDpForEtaleGroupoids} in several ways. First of all, étale groupoids with polynomial growth all have strong subexponential growth, but the converse is not true. It therefore covers a strictly larger class of examples. Secondly, it does not require the length function to be continuous, only locally bounded. Finally, and more subtly, the case $p = 1$ is included in this result.  It was out of necessity left out of \cite[Theorem 4.7]{AustadOrtegaPalmstrom:PolynomialGrowthAndPropertyRDpForEtaleGroupoids} due to its proof relying in part on duality arguments, an issue we remedy in this article by employing a novel approach using interpolation techniques from \cite{OztopSameiShepelska:TwistedOrliczAlgebrasAndCompleteIsomorphismToOperatorAlgebras, Samei&Shepelska:NormcontrolledInverseionInWeightedConvolutionAlgebras} to construct spectrally invariant subalgebras of the reduced $L^p$-operator algebras. Note that the construction of these algebras are new even in the $C^*$-algebra setting. 
By opting to use interpolation techniques, the proof of \cref{introthm: strong subexponential growth gives invariance in K-theory} is also both simpler and shorter than the proof of \cite[Theorem 4.7]{AustadOrtegaPalmstrom:PolynomialGrowthAndPropertyRDpForEtaleGroupoids}.

As mentioned, every étale groupoid with polynomial growth will have strong subexponential growth, and it can be observed that this class containment is strict.
%This 
Indeed, this is even the case for discrete groups since it is well known that the Grigorchuck group is an example of a group which has strong subexponential growth but not polynomial. Also, any action of the Grigorchuck group on a locally compact Hausdorff space yields an étale groupoid with the same growth properties. To obtain an example not related to groups, there are many metric spaces with intermediate growth arising from graph theory (see for example \cite{BondarenkoEtAl2012, KontogeorgiouWinter2022, Lehner2016, MiasnikovSavchuk2015}), and then by \cref{cor:growth-coarse-groupoids} the associated coarse groupoid has strong subexponential growth, but not polynomial. Establishing that the Grigorchuck group has intermediate growth %was a major accomplishment of Grigorchuck, solving at the time the important 
solved a big open question of whether or not there exist groups of intermediate growth. Using classical tools from analytic number theory as well as recent results by Brix, Hume and Li in \cite{BrixHumeLi:MinimalCovers}, we are able to provide an example of an étale groupoid with intermediate growth by comparatively much simpler means. The groupoid is a Renault-Deaconu-type groupoid associated with a certain shift that we call the \emph{ordered prime shift} (see \cref{subsec: shift groupoid}). This leads us to our second main result.

\begin{introtheorem} [cf. \cref{thm: étale groupoid built from prime shift}] \label{introthm: non-trivial example of an intermediate growth groupoid}
	Let $\G_X$ denote the Renault-Deaconu groupoid associated with the ordered prime shift space $X$, and let $\hat{\G}_X$ denote its cover groupoid. Then $\hat{\G}_X$ is an étale groupoid that has strong subexponential growth, but not polynomial growth.
\end{introtheorem}

The paper is structured as follows. In \cref{sec: preliminaries}, we recall some definitions and results regarding étale groupoids and their associated reduced $L^p$-operator algebras. In \cref{sec: growth of étale groupoids} we introduce the notion of (strong) subexponential growth for étale groupoids and prove some properties such groupoids necessarily must possess.  \cref{sec: applications} covers the proof of our main result \cref{introthm: strong subexponential growth gives invariance in K-theory}, but we also prove a similar result for the symmetrized versions of the reduced groupoid $L^p$-operator algebras therein. %The final section is 
Lastly, in \cref{sec: examples} %, wherein 
we exhibit several examples of étale groupoids with strong subexponential growth, the most significant % important 
of which is the example of the ordered prime shift in \cref{subsec: shift groupoid}.

\section{Preliminaries} \label{sec: preliminaries}
We recall some basic terminology and results regarding étale groupoids and their reduced $L^p$-operator algebras. The reader is referred to \cite{ChoiGardellaThiel:RigidityResultsForLpOperatorAlgebrasAndApplications, Renault:AGroupoidApproach} for details.

Let $\G$ be a locally compact Hausdorff groupoid with unit space $\G^{(0)}$, composable pairs $\G^{(2)}$, and range and source maps $r,s \colon \G \to \G$ given by $r (x) = xx^{-1}$ and $s(x) = x^{-1}x$. 
The groupoid $\G$ is called \emph{étale} if the range and source maps are local homeomorphisms. 
For any $X \subset \G^{(0)} ,$ we denote by $\G_X = \{x \in \G \colon s(x) \in X\} $, $\G^{X} = \{x \in \G \colon r(x) \in X\}$, and $\G(X) = \G_X \cap \G^X$. We shall write $\G_u$ and $\G^{u}$ instead of $\G_{\{u\}}$ and $\G^{\{u\}}$, whenever $u \in \G^{(0)}$ is a unit. When $\G$ is étale, the fibers $\G_u$ and $\G^u$, for $u \in \G^{(0)}$, are discrete.
The \emph{isotropy group} at a unit $u \in \G^{(0)}$ is the group $\G(u) := \G_u \cap \G^{u}$, and the \emph{isotropy bundle} is the set $\mathrm{Iso}(\G) = \bigsqcup_{u \in \G^{(0)}} \G(u)$. If $\mathrm{Iso}(\G)^{\circ} = \G^{(0)}$, then $\G$ is said to be \emph{effective}. Given subsets $A,B \subset \G$, their \emph{product} is the set $AB = \{ ab \mid (a,b) \in \G^{(2)}, a \in A, b \in B \}$.

A \emph{length function} on a groupoid $\G$ is a map $\lf \colon \G \to \R_{\geq 0}$ such that $\lf(u) = 0$, for all $u \in \G^{(0)}$, $\lf(x^{-1}) = \lf(x)$, for all $x \in \G$, and $\lf(xy) \leq \lf(x) + \lf(y)$, whenever $(x,y) \in \G^{(2)}$. We say a length function $\lf$ is \emph{locally bounded} if it is bounded on compact sets. 
Two natural examples of length functions on étale groupoids are the following: First, suppose $\Gamma$ is a discrete group with a length function $\lf_\Gamma \colon \Gamma \to \R_{\geq 0}$, and suppose that $\Gamma$ acts on a locally compact space $X$. Then it is easy to see that $\lf (\gamma , x) = \lf_\Gamma (\gamma)$ is a locally bounded length function on the transformation groupoid $\Gamma \ltimes X$. Second, if $\G$ is an étale groupoid which is compactly generated in the sense that there is a compact set $S \subset \G$ such that $S = S^{-1}$ and $\G = \bigcup_{k = 1}^{\infty} S^k$, then the function $\lf_S$ given by $\lf_S (u) := 0$, for all units $u \in \G^{(0)}$, and for $x \in \G \setminus \G^{(0)}$, $ \lf_S (x) = \inf\{ k \mid x \in S^k \} $, is a length function on $\G$. $\lf_S$ is locally bounded if for example also $\G = \bigcup_{k = 1}^{\infty} (S^k)^\circ$.

Now let $C_c (\G)$ denote the space of compactly supported continuous functions on $\G$. We endow $C_c (\G)$ with the convolution product, which for $f,g \in C_c (\G)$ is given by $$ f \ast g (x) = \sum_{y \in \G_{s(x)}} f(x y^{-1}) g(y) = \sum_{y \in \G^{r(x)}} f(y) g(y^{-1}x) ,$$ for $x \in \G$, and involution given by $$ f^{\ast}(x) = \overline{f(x^{-1})} ,$$ for $f \in C_c (\G)$ and $x \in \G$. The \emph{$I$-norm} on $C_c (\G)$ is given by $$ \| f \|_I = \max \left\lbrace \sup_{u \in \G^{(0)}} \sum_{x \in \G_u} |f(x)| \, , \, \sup_{u \in \G^{(0)}} \sum_{x \in \G^{u}} |f(x)| \right\rbrace .$$ With the above norm and algebraic operations, $(C_c (\G), \ast, ^{\ast}, \| \cdot \|_{I})$ is a normed $^*$-algebra. 
Let $p \in [1, \infty)$ and fix any unit $u \in \G^{(0)}$. The operator $\lambda_u (f) \in B(\ell^p (\G_u))$ associated with $f \in C_c(\G)$, is the operator given by 
\begin{equation*}
	\lambda_u (f) (\xi) (x) = \sum_{y \in \G_u} f(x y^{-1}) \xi (y) ,
\end{equation*}
for $x \in \G_{u}$ and $\xi \in \ell^p (\G_u)$. The map $\lambda_u \colon C_c (\G) \to B(\ell^p (\G_u))$ is an $I$-norm contractive homomorphism of $C_c(\G)$, and is called the \emph{left regular representation} at $u$. The \emph{reduced groupoid $L^p $-operator algebra} associated with $\G$ is denoted $F_{\lambda}^p (\G)$ and is the completion of $C_c(\G)$ under the norm $$\lVert f \rVert_{\Fp} := \sup_{u \in \G^{(0)}} \lVert \lambda_{u}(f) \rVert .$$ 
By \cite[Lemma 4.5]{ChoiGardellaThiel:RigidityResultsForLpOperatorAlgebrasAndApplications} this norm satisfies the following, for any $f \in C_c (\G)$, $$ \lVert f \rVert_{\infty} \leq \lVert f \rVert_{\Fp} \leq \lVert f \rVert_I . $$ 
Since $\bigoplus_{u \in \G^{(0)}} \lambda_u$ is an isometric representation of $F_{\lambda}^p (\G)$ on the $L^p $-space $\bigoplus_{u \in \G^{(0)}} \ell^p (\G_u)$, $F_{\lambda}^p (\G)$ is an $L^p $-operator algebra. It is unital if and only if $\G^{(0)}$ is compact, in which case the indicator function of the unit space is the identity element.
The map $j_p \colon F_{\lambda}^{p}(\G) \to C_0 (\G)$ given by 
\begin{equation*}
	j_p (a) (x) = \lambda_{s(x)} (a) (\delta_{s(x)}) (x) ,
\end{equation*}
for $a \in F_{\lambda}^p (\G)$ and $x \in \G$, is contractive, linear, injective and extends the identity map on $C_c (\G)$. We shall refer to this map as \emph{Renault's p-j-map}. Given $a,b \in \Fp(\G)$ and $x \in \G$, we have that $j_p (a b) (x) = j_p (a) \ast j_p (b) (x)$, where the sum defining $j_p (a) \ast j_p (b) (x)$ is absolutely convergent (see \cite[Proposition 4.7 and Proposition 4.9]{ChoiGardellaThiel:RigidityResultsForLpOperatorAlgebrasAndApplications}).

\section{Growth of étale groupoids} \label{sec: growth of étale groupoids}

Let $\lf \colon \G \to [0,\infty)$ be a locally bounded length function on an étale groupoid $\G$. Given $t \geq 0$, we define $$ B_{\G_u}(t) := \{ x \in \G_u \mid \lf(x) \leq t \} ,$$ and define $B_{\G^u}(t)$ analogously. Since $\lf$ is symmetric, $|B_{\G_u}(t)| = |B_{\G^u}(t)|$, for every $t \geq 0$. 

\begin{definition} \label{def: (strongly) subexponential growth}
	Let $\G$ be an étale groupoid endowed with a locally bounded length function $\lf$. 
	We say that $\G$ has \emph{polynomial growth} with respect to $\lf$ if there exist constants $C > 0$ and $d \in \N$ such that $$ \sup_{u \in \G^{(0)}} |B_{\G_u} (t)| \leq C (1+t)^d ,$$ for all $t \geq 0$;
	\emph{strong subexponential growth} with respect to $\lf$ if there exist $\alpha > 0$, $0 < \beta < 1$ and a positive constant $C > 0$ such that $$ \sup_{u \in \G^{(0)}} |B_{\G_u} (t)| \leq C \exp(\alpha t^\beta) ,$$ for each $t \geq 0$;
	\emph{subexponential growth} with respect to $\lf$ if $$\limsup_{t \to \infty} \sup_{u \in \G^{(0)}} |B_{\G_u} (t)|^{1/t} = 1 .$$ 
	If $\G$ does not have subexponential growth with respect to $\lf$, then we say $\G$ has \emph{exponential growth} with respect to $\lf$. 
\end{definition}

Clearly we have the following relationship between the above definitions: 
$$ \text{polynomial growth} \implies \text{strong subexponential growth} \implies \text{subexponential growth} .$$ 
In \cite{AustadOrtegaPalmstrom:PolynomialGrowthAndPropertyRDpForEtaleGroupoids} the authors exhibited several examples of étale groupoids with polynomial growth, for example AF-groupoids, certain point set groupoids, coarse groupoids associated with uniformly locally finite metric spaces with polynomial growth, and Renault-Deaconu groupoids associated with certain finite directed graphs. To obtain an example of an étale groupoid with strong subexponential growth which does not have polynomial growth, one can simply take the transformation groupoid formed from any action of the Grigorchuck group on a locally compact space. Equipped with the natural length function described in \cref{sec: preliminaries}, the growth of the transformation groupoid is the same as that of the Grigorchuck group, which has strong subexponential growth (see \cite[Theorem B 1.]{Grigorchuk:DegreesOfGrowth}). Also, in \cref{subsec: shift groupoid}, we shall exhibit an example of an étale groupoid associated with a shift of infinite type that has strong subexponential growth, but not polynomial growth. 

Following \cite[Definition 5.4 and Proposition 5.5]{MaWuAlmostElementarinessAndFiberwiseAmenabilityForEtaleGroupoids} we say that an étale groupoid $\G$ is \emph{fiberwise amenable} if for any compact $K \subset \G$  and $\epsilon > 0$, there exists $F \subset \G$ finite such that $$ \left| K F \right| / \left| F \right| \leq 1 + \epsilon .$$ 
Fiberwise amenability was introduced by Ma and Wu in \cite{MaWuAlmostElementarinessAndFiberwiseAmenabilityForEtaleGroupoids}. For certain classes of groupoids, it is a strengthening of the notion of amenability for groupoids. Indeed, in the case of transformation groupoids, it is equivalent to the acting group being amenable, while in the case of coarse groupoids associated with metric spaces, it is equivalent to metric amenability of the underlying metric space. In general, however, there is no implication between the two notions. 

\begin{proposition} \label{prop: Subexponential groupoids are fiberwise amenable}
	If $\G$ is an étale groupoid with subexponential growth with respect to a locally bounded length function, then $\G$ is fiberwise amenable.
	\begin{proof}
		Let $\lf \colon \G \to \R_{\geq 0}$ be the length function, and let $K \subset \G$ be compact. Then $\lf(K) \subseteq [0,M]$ for some $M \in \N$. Fix any unit $u \in \G^{(0)}$ and let $\epsilon > 0$ be given; then we may find $k \in \N_0$ such that $$ \left| B_{\G_u} (k+M) \right| / \left| B_{\G_u} (k) \right| \leq 1 + \epsilon .$$ Indeed, if this were not the case, then for every $k \in \N_0$, we would have $$ |B_{\G_{u}}(k+M)| > (1+\epsilon) |B_{\G_{u}}(k)| .$$ In particular, by induction, $$ |B_{\G_{u}}(kM) | > (1+\epsilon)^{k} , $$ for every $k \in \N$, from which we get the contradiction that $$ \limsup_{k \to \infty} \sup_{u \in \G^{(0)}} |B_{\G_{u}}(k)|^{1/k} \geq \limsup_{k \to \infty} |B_{\G_{u}}(k)|^{1/k} > 1 .$$ Also, since $\G$ has subexponential growth, $\left| B_{\G_u} (k) \right| < \infty$, and combining this with the observation $$ \left| K B_{\G_u} (k) \right| / \left| B_{\G_u} (k) \right| \leq \left| B_{\G_u} (k+M) \right| / \left| B_{\G_u} (k) \right| \leq 1 + \epsilon ,$$ we see that $\G$ is indeed fiberwise amenable.
	\end{proof}
\end{proposition}

Fiberwise amenability has connections with soficity of the topological full groups (see \cite[Section 7]{Ma:FiberwiseAmenabilityOfAmpleGroupoids}), and, more importantly to us, for $\sigma$-compact groupoids with compact unit space, it implies the existence of invariant probability measures on the unit space (see \cite[Proposition 5.9]{MaWuAlmostElementarinessAndFiberwiseAmenabilityForEtaleGroupoids}).

\begin{example} \label{ex: a class of examples of exponential growth groupoids}
	Suppose $\G$ is a locally compact, $\sigma$-compact, Hausdorff, étale, minimal and effective groupoid with compact unit space. Then $C^*_r(\G)$ is a simple and unital $C^*$-algebra. If $\G$ has subexponential growth, then by \cref{prop: Subexponential groupoids are fiberwise amenable}, $\G$ is fiberwise amenable, and hence by \cite[Proposition 5.9]{MaWuAlmostElementarinessAndFiberwiseAmenabilityForEtaleGroupoids} there exists a $\G$-invariant Borel probability measure $\mu$ on $\G^{(0)}$. Then $\tau:=\mu \circ E$, where $E:C^*_r(\G)\to C(\G^{(0)})$ is the canonical conditional expectation, is a faithful trace. If $\Cred(\G)$ is purely infinite, however, then it has no faithful trace. In particular, a $\sigma$-compact étale groupoid $\G$ which is minimal, effective, with compact unit space and for which $\Cred(\G)$ is purely infinite must have exponential growth with respect to any locally bounded length function. 
\end{example}

\section{Applications to k-theory of groupoid \( L^p \)-operator algebras} \label{sec: applications}
Our aim for this section is to show that the K-groups $K_\ast (F_{\lambda}^{p}(\G))$ are independent of $p \in [1, \infty)$, for $\ast = 0,1$, whenever $\G$ is an étale groupoid which has strong subexponential growth with respect to a locally bounded length function.

Throughout, $\G$ will always denote a fixed étale groupoid endowed with a fixed locally bounded length function $\lf \colon \G \to \R_{\geq 0}$ and we fix some $p \in [1, \infty)$.
Given $\alpha, \alpha^\prime > 0$ and $\beta, \beta^\prime \in (0,1)$, let us say $(\alpha, \beta) \geq (\alpha^\prime, \beta^\prime)$ if $\alpha \geq \alpha^\prime$ and $\beta \geq \beta^\prime$. For ease of notation, whenever we write a pair $(\alpha , \beta)$ we shall always implicitly assume that $\alpha > 0$ and $\beta \in (0,1)$.
Notice that if the condition in the definition of strong subexponential growth is satisfied for some pair $(\alpha^\prime, \beta^\prime)$, then it is satisfied for all $(\alpha, \beta) \geq (\alpha^\prime, \beta^\prime)$.

For each pair $(\alpha, \beta)$, let $\omega_{\alpha, \beta} \colon \G \to \R_{\geq 0}$ be the function given by $ \omega_{\alpha, \beta}(x) := \exp(\alpha \lf(x)^\beta)$. Then $\omega_{\alpha, \beta}$ is a submultiplicative weight on $\G$, meaning that $\omega_{\alpha, \beta}(u) = 1$ for each $u \in \G^{(0)}$, $\omega_{\alpha, \beta}(x^{-1}) = \omega_{\alpha, \beta}(x)$, and $\omega_{\alpha, \beta}(xy) \leq \omega_{\alpha, \beta}(x)\omega_{\alpha, \beta}(y)$, whenever $(x,y) \in \G^{(2)}$. 

For any complex-valued function $f$ on $\G$ and $q \in [1, \infty)$ we define $$ \| f \|_q := \max \left\lbrace \sup_{u \in \G^{(0)}} \left\lbrace \sum_{x \in \G_u} |f(x)|^q \right\rbrace^{1/q} , \sup_{u \in \G^{(0)}} \left\lbrace \sum_{x \in \G^u} |f(x)|^q \right\rbrace^{1/q} \right\rbrace ,$$ while as usual, $\| f \|_\infty = \sup_{x \in \G} |f(x)|$. We let $ \ell^q (\G) := \{ f \colon \G \to \C \mid \| f \|_q < \infty \} $, for $q \in [1, \infty]$. The next lemma will be used frequently throughout.

\begin{lemma} \label{lem: subexponential growth gives norm estimate of exponential}
	If $\G$ has strong subexponential growth with respect to $\lf$ with %and pair of 
	constants $(\alpha_0 , \beta_0)$ and $C>0$ as in \cref{def: (strongly) subexponential growth}, then for every $(\alpha , \beta) \geq (\alpha_0 , \beta_0)$ with $\alpha > \alpha_0$, one has that $$ \sup_{u \in \G^{(0)}} \sum_{x \in \G_u} \exp(-\alpha \lf(x)^\beta) < \infty .$$
	\begin{proof}
		Fix any $u \in \G^{(0)}$. We compute
		\begin{align*}
			\sum_{x \in \G_u} \exp(-\alpha \lf(x)^\beta) &= \sum_{k = 0}^{\infty} \sum_{\substack{x \in \G_u \\ k \leq \lf(x) \leq k+1}} \exp(-\alpha \lf(x)^\beta) \\
			&\leq \sum_{k = 0}^{\infty} |B_{\G_u}(k+1)| \exp(-\alpha k^\beta) \\
			&\leq \sum_{k = 0}^{\infty} C \exp(\alpha_0 (k+1)^{\beta_0}) \exp(-\alpha k^\beta) \\
			&\leq \exp(\alpha_0) C \sum_{k = 0}^{\infty} \exp(-(\alpha - \alpha_0) k^\beta) < \infty ,
		\end{align*}
		independently of $u \in \G^{(0)}$, since $\alpha > \alpha_0 $. Note that in the last transition we used that $(k+1)^{\beta_0} \leq k^{\beta_0} + 1$ as $0 < \beta_0 < 1$. The result follows.
	\end{proof}
\end{lemma}

Given $\alpha > 0$ and $\beta \in (0,1)$, let $\| \cdot \|_{\alpha , \beta}$ denote the norm on $C_c (\G)$ given by $$ \| f \|_{\alpha, \beta} := \| f \omega_{\alpha , \beta} \|_p ,$$ and let $\Sab(\G) = \overline{C_c (\G)}^{\| \cdot \|_{\alpha, \beta}}$ be the Banach space obtained by completing $C_c (\G)$ in this norm. 

Since we have the norm inequality $\| \cdot \|_{\infty} \leq \| \cdot \|_{\alpha, \beta}$ on $C_c (\G)$, it follows by a similar argument to \cite[Lemma 3.4]{AustadOrtegaPalmstrom:PolynomialGrowthAndPropertyRDpForEtaleGroupoids} that $\Sab(\G) \subset C_0 (\G)$.

\begin{proposition} \label{prop: norm inequality for reduced Lp norm and application to inverse of Renaults j-p-map}
	If $\G$ has strong subexponential growth with respect to $\lf$ and pair of constants $(\alpha_0 , \beta_0)$, then for each pair $(\alpha , \beta) \geq (\alpha_0 , \beta_0)$ with $\alpha > \alpha_0$, there exists a constant $K_{\alpha, \beta} > 0$ such that $$ \| f \|_{\Fp} \leq K_{\alpha, \beta} \| f \|_{\alpha, \beta} ,$$ for all $f \in C_c (\G)$. In particular, the inverse of Renault's p-j-map gives a linear, injective and continuous map $j_{p}^{-1} \colon \Sab (\G) \to \Fp (\G)$. 
	\begin{proof}
		The first statement follows from the fact that $$\| f \|_{\Fp} \leq \| f \|_I \leq K_{\alpha , \beta} \| f \|_{\alpha , \beta} ,$$ for every $f \in C_c (\G)$, by an application of Hölder's inequality together with \cref{lem: subexponential growth gives norm estimate of exponential}. To see the second statement, let $\iota \colon C_c (\G) \to \Fp (\G)$ be the usual inclusion. Having proved the first statement here, it follows that the inclusion extends to a bounded linear map $\iota_{\alpha , \beta} \colon \Sab (\G) \to \Fp (\G)$. If now $f \in \Sab (\G)$, find $f_n \to f$ in $\Sab$, where $f_n \in C_c (\G)$. Then $ \iota_{\alpha, \beta}(f_n) \to \iota_{\alpha, \beta}(f)$ in $\Fp (\G)$, so that $$j_p (\iota_{\alpha, \beta} (f)) = \lim_n j_p (\iota_{\alpha, \beta} (f_n)) = \lim_n f_n = f.$$ It follows that $\iota_{\alpha, \beta} (f) = j_{p}^{-1}(f)$.
	\end{proof}
\end{proposition}

\begin{remark} \label{rmk: may assume the length function is integer valued}
	If $\lf$ is any locally bounded length function with respect to which $\G$ has (strong) subexponential growth, then one can obtain a locally bounded integer-valued length function with respect to which $\G$ also has (strong) subexponential growth; simply define $\tilde{\lf}(x) = \ceil(\lf(x))$, where for $t \geq 0$, $\ceil(t)$ is the smallest integer larger than or equal to $t$. Then clearly $\lf \leq \tilde{\lf} \leq 1 + \lf$, and from this it follows that $\G$ has (strong) subexponential growth with respect to $\tilde{\lf}$, and moreover that the norms $\| \cdot \|_{\alpha, \beta}$ defined in terms of $\lf$ are equivalent to the ones defined in terms of $\tilde{\lf}$. 
\end{remark}

Let us for the remainder of this section assume $\G$ has strong subexponential growth with respect to $\lf$, with constants $C > 0$ and $(\alpha_0 , \beta_0)$ as in \cref{def: (strongly) subexponential growth}. For our purposes, we may assume by \cref{rmk: may assume the length function is integer valued} that $\lf$ takes integer values. We also fix a pair $(\alpha, \beta) \geq (\alpha_0 , \beta_0)$ with $\alpha > \alpha_0$, and consider the associated Banach space $\Sab(\G)$. For ease of notation, we write $\omega$ instead of $\omega_{\alpha, \beta}$ for the associated weight. Define two auxiliary functions $u$ and $\sigma$ as follows: 
$$ u(x) := \exp(- \alpha (2-2^{\beta}) \lf(x)^{\beta}) ,$$ for $x \in \G$, and $\sigma := \omega u$. The following inequality is essential to the arguments in this section, and can be proved exactly like in \cite[Theorem 2.2]{OztopSameiShepelska:TwistedOrliczAlgebrasAndCompleteIsomorphismToOperatorAlgebras}:
\begin{equation} \label{eq: important equation involving weights and auxilliary function}
	\frac{\omega(xy)}{\omega(x) \omega(y)} \leq u(x) + u(y) ,
\end{equation}
for all $(x,y) \in \G^{(2)}$.

We also need the following Young's convolution inequality for étale groupoids. We omit its proof because it is completely analogous to that in the setting of discrete groups.

\begin{lemma} \label{lem: Young's convolution inequality for groupoids}
	Let $p,q,r \geq 1$ be such that $1 + 1/r = 1/p + 1/q$. If $f \in \ell^p (\G)$ and $g \in \ell^q (\G)$, then $f \ast g \in \ell^r (\G)$ and $$ \| f \ast g \|_r \leq \| f \|_p \| g \|_q .$$ 
\end{lemma}

\begin{lemma} \label{lem: inequality following from Young's}
	If $f, g \in \Sab(\G)$, then $$ \| f \ast g \|_{\alpha, \beta} \leq \| f \|_{\alpha, \beta} \| g \sigma \|_1 + \| f \sigma \|_1 \|g \|_{\alpha, \beta} .$$
	\begin{proof}
		It follows by \cref{eq: important equation involving weights and auxilliary function} that $$ | f \ast g (x) \omega(x) | \leq |f| \omega \ast |g | \sigma (x) + |f | \sigma \ast |g | \omega (x) ,$$ for each $x \in \G$, and therefore $$ \| f \ast g \|_{\alpha, \beta} \leq \| |f| \omega \ast |g | \sigma \|_p + \| |f | \sigma \ast |g | \omega \|_p \leq \| f \|_{\alpha, \beta} \| g \sigma \|_1 + \| f \sigma \|_1 \| g \|_{\alpha, \beta} ,$$ by \cref{lem: Young's convolution inequality for groupoids}.
	\end{proof}
\end{lemma}

\begin{lemma} \label{lem: interpolation inequality}
	Suppose $(\alpha , \beta) \geq (\alpha_0 , \beta_0)$ is such that $\alpha (2 - 2^\beta) > \alpha_0$. Then there exists $\theta = \theta(\alpha, \beta) \in (0,1)$ and a constant $C_{\alpha, \beta} > 0$, such that $$ \| f \ast f\|_{\alpha, \beta} \leq C_{\alpha, \beta} \| f \|_{\alpha, \beta}^{1+ \theta} \| f \|_{\Fp}^{1-\theta} ,$$ for all $f \in \Sab(\G)$. 
	\begin{proof}
		%We first make the following observation: 
		For $\theta \in (0,1)$, let $$\xi_\theta (x) := \omega^{1-\theta}(x) u(x) = \exp(\alpha(1- \theta) \lf(x)^{\beta} - \alpha (2 - 2^\beta) \lf(x)^{\beta}) =\exp( \lf(x)^{\beta} \alpha ( 1- \theta - 2 + 2^\beta ) ) .$$ %This may be rewritten to $\xi_\theta (x) = \exp( \lf(x)^{\beta} \alpha ( 1- \theta - 2 + 2^\beta ) )$. 
		Now, as $2 > 2^{\beta}$ since $\beta \in (0,1)$ and by assumption $\alpha (2 - 2^\beta) > \alpha_0$, there exists $\theta \in (0,1)$ such that $2 + \theta > 2^\beta + 1$ and $\alpha (\theta - 1 + 2 - 2^\beta) > \alpha_0$. For such a $\theta$, we have that $\| \xi_\theta \|_s < \infty$ for all $s \in [1,\infty]$, by \cref{lem: subexponential growth gives norm estimate of exponential}. 
		
		Now let $f \in \Sab(\G)$. By \cref{lem: inequality following from Young's} with $f = g$, we have that $$ \| f \ast f \|_{\alpha, \beta} \leq 2 \| f \sigma \|_1 \| f \|_{\alpha, \beta} .$$ 
		
		Let $u \in \G^{(0)}$ be any unit. Then,
		\begin{align*}
			\| f \sigma \|_{\ell^1 (\G_u)} &= \sum_{x \in \G_u} |f(x)| \sigma (x) \\
			&= \sum_{x \in \G_u} |f(x)| \omega(x) u(x) \\
			&= \sum_{x \in \G_u} |f(x)|^{1- \theta} (f(x) \omega(x))^{\theta} \xi_\theta \\
			&\leq \left\lbrace \sum_{x \in \G_u} |f(x)|^p \right\rbrace^{\frac{1-\theta}{p}} \left\lbrace \sum_{x \in \G_u} |f(x) \omega(x)|^p  \right\rbrace^{\frac{\theta}{p}} \left\lbrace \sum_{x \in \G_u} \xi_{\theta}^q \right\rbrace^{\frac{1}{q}} \\
			&\leq \| f \|_{F_{\lambda}^p}^{1-\theta} \| f \|_{\alpha, \beta}^{\theta} \| \xi_\theta \|_q , 
		\end{align*}
		by the generalized Hölder inequality with the exponents $\frac{p}{1-\theta}$, $\frac{p}{\theta}$ and $q$. Also,
		\begin{align*}
			\| f \sigma \|_{\ell^1 (\G^u)} &= \| f^\ast \sigma \|_{\ell^1 (\G_u)} \\
			&\leq \left\lbrace \sum_{x \in \G_u} |f^\ast (x)|^q \right\rbrace^{\frac{1-\theta}{q}} \left\lbrace \sum_{x \in \G_u} |f^\ast (x) \omega(x)|^p  \right\rbrace^{\frac{\theta}{p}} \left\lbrace \sum_{x \in \G_u} \xi_{\theta}^{s} \right\rbrace^{\frac{1}{s}} \\
			&\leq \| f \|_{F_{\lambda}^p}^{1-\theta} \| f \|_{\alpha, \beta}^{\theta} \| \xi_\theta \|_s
		\end{align*}
		where $1 < s = s(\theta)$ is such that $$ \frac{1-\theta}{q} + \frac{\theta}{p} + \frac{1}{s} = 1 .$$ In the third inequality in the preceding computation, we used the fact that the involution on $C_c (\G)$ extends to an isometric anti-isomorphism $\Fp(\G) \to F_{\lambda}^{q}(\G)$ in the case $p \in (1, \infty)$ (this follows for example by \cite[Lemma 3.5]{AustadOrtega:GroupoidsAndHermitianBanachstarAlgebras}), and when $p = 1$, that $\| f \|_\infty \leq \| f \|_{\Fp}$, for $f \in \Sab(\G)$. By our observation in the beginning, we may find $\theta$ such that $\xi_\theta$ is in $\ell^s (\G)$ for every $s \in [1, \infty]$. Choose such a $\theta$,  find $s = s(\theta)$, and put $C_{\alpha, \beta} := 2 \max \{ \| \xi_\theta \|_s , \| \xi_\theta \|_q \} < \infty$. Then $$ \| f \sigma \|_1 \leq \frac{C_{\alpha,\beta}}{2}\cdot  \| f \|_{\Fp}^{1- \theta} \| f \|_{\alpha, \beta}^{\theta}, $$ so that $$ \| f \ast f \|_{\alpha, \beta} \leq C_{\alpha, \beta} \| f \|_{\Fp}^{1- \theta} \| f \|_{\alpha, \beta}^{1+\theta} ,$$ as desired.
	\end{proof}
\end{lemma}

%The following definition will be used extensively throughout the remainder of this section.
%\begin{definition}
%	For $(\alpha_0, \beta_0) \in \R_{>0} \times (0,1)$ define the index set 
%\end{definition}
Let us define the index set $$ J(\alpha_0 , \beta_0) := \{ (\alpha , \beta) \in \R_{> 0} \times (0,1) \colon (\alpha , \beta) \geq (\alpha_0 , \beta_0) \text{ and } \alpha (2 - 2^\beta) > \alpha_0 \} .$$

\begin{proposition} \label{prop: weighted Lp-spaces are Banach algebras}
	If $(\alpha, \beta) \in J(\alpha_0 , \beta_0)$, then $\Sab(\G)$ is a Banach algebra.
	\begin{proof}
		Let $f,g \in \Sab(\G)$. By \cref{lem: inequality following from Young's}, we have that $$ \| f \ast g \|_{\alpha, \beta} \leq \| f \|_{\alpha , \beta} \| g \sigma \|_1 + \| f \sigma \|_1 \|g \|_{\alpha, \beta} ,$$ and in the proof of \cref{lem: interpolation inequality} we saw that $$ \| f \sigma \|_1 \leq C_{\alpha, \beta} \| f \|_{\Fp}^{1-\theta} \| f \|_{\alpha, \beta}^{\theta} ,$$ for some positive constant $C_{\alpha, \beta}$ and $\theta \in (0,1)$. By \cref{prop: norm inequality for reduced Lp norm and application to inverse of Renaults j-p-map}, $$ \| f \|_{\Fp} \leq K_{\alpha, \beta} \| f \|_{\alpha, \beta} ,$$ for all $f \in C_c (\G)$ and a constant $K_{\alpha, \beta}$, and by continuity of the inclusion, it also holds for all $f \in \Sab(\G)$. Putting all of this together, we see that 
		$$ \|f \ast g \|_{\alpha, \beta} \leq C_{\alpha, \beta}^{\prime} \| f \|_{\alpha, \beta} \| g \|_{\alpha, \beta} ,$$ for some constant $C_{\alpha, \beta}^{\prime} > 0$ not depending on $f$ and $g$. It follows that $\Sab(\G)$ is a Banach algebra.
	\end{proof}
\end{proposition}

\begin{proposition} \label{prop: weighted Lp spaces are spectral invariant in reduced groupoid Lp}
	For any $(\alpha, \beta) \in J(\alpha_0 , \beta_0)$, the Banach algebra $\Sab(\G)$ is spectrally invariant in $F_{\lambda}^{p} (\G)$. 
	\begin{proof}
		Recall from \cref{lem: interpolation inequality} that there exist $C_{\alpha, \beta} > 0$ and $\theta \in (0,1)$ such that $$ \| f \ast f\|_{\alpha, \beta} \leq C_{\alpha, \beta} \| f \|_{\alpha, \beta}^{1+ \theta} \| f \|_{\Fp}^{1-\theta} ,$$ for all $f \in \Sab(\G)$. In particular, $$ \| f^{2n} \|_{\alpha, \beta} = \| (f^n)^2 \|_{\alpha, \beta} \leq C_{\alpha, \beta} \| f^n \|_{\alpha, \beta}^{1+ \theta} \| f^n \|_{\Fp}^{1-\theta} .$$ Raising everything to the power $1/2n$ and taking limits gives $$ r_{\Sab(\G)}(f) \leq r_{\Sab(\G)} (f)^{\frac{1+\theta}{2}} r_{F_{\lambda}^{p}(\G)} (f)^{\frac{1-\theta}{2}} ,$$ which is equivalent to $r_{\Sab(\G)} (f) \leq r_{F_{\lambda}^{p}(\G)} (f)$. Since the other inequality follows by the inclusion $\Sab(\G) \hookrightarrow F_{\lambda}^{p}(\G)$, we see that $r_{\Sab(\G)} (f) = r_{F_{\lambda}^{p}(\G)} (f)$, for all $f \in \Sab(\G)$. This suffices by \cite[Lemma 2.7]{LauterMonthubertNistor:SpectralInvarianceForCertainAlgebrasOfPseudodifferentialOperators}.
	\end{proof}
\end{proposition}

%Since for Banach algebras, spectral invariance is equivalent to stability under holomorphic functional calculus by \cite[Lemma 1.2]{Schweitzer:AShortProofThatMnAIsLocal}, \cite[Chapter 3, Appendix C, Proposition 3]{Connes:NoncommutativeGeometry} applies to give the following result.
Since spectral invariance is equivalent to stability under holomorphic functional calculus for Banach algebras by \cite[Lemma 1.2]{Schweitzer:AShortProofThatMnAIsLocal}, \cite[Chapter 3, Appendix C, Proposition 3]{Connes:NoncommutativeGeometry} applies to give the following result.

\begin{corollary} \label{cor: weighted Lp spaces all capture K-theory}
	For any $(\alpha, \beta) \in J(\alpha_0 , \beta_0)$, $K_\ast (\Sab(\G)) \cong K_\ast (F_{\lambda}^{p}(\G))$, where $\ast = 0,1$.
\end{corollary}

Recall that $(\alpha, \beta) \geq (\alpha^\prime, \beta^\prime)$ if $\alpha \geq \alpha^\prime$ and $\beta \geq \beta^\prime$. In this case, since $\| \cdot \|_{\alpha, \beta} \geq \| \cdot \|_{\alpha^\prime, \beta^\prime}$, it follows that $\Sab(\G) \subset S_{\lf, (\alpha^\prime , \beta^\prime)}^{p}(\G)$. Therefore, $\{ \Sab(\G) \}_{(\alpha, \beta) \in J(\alpha_0 , \beta_0)}$ forms a directed system of decreasing Banach algebras. 

\begin{definition} \label{def: Fréchet algebra of subexponentially decreasing functions}
	Let $\G$ be an étale groupoid with strong subexponential growth with respect to the locally bounded length function $\lf$ and constants $\alpha_0 > 0$ and $0 < \beta_0 < 1$ as in \cref{def: (strongly) subexponential growth}. We define the space of \emph{strongly subexponentially decreasing functions} as $$ S_{\lf}^{p} (\G) := \bigcap_{ (\alpha, \beta) \in J(\alpha_0 , \beta_0) } \Sab(\G) .$$ Since the system $\{ (\alpha, \beta) \}_{(\alpha, \beta) \in J(\alpha_0 , \beta_0) }$ has a countable cofinal sequence, $S_{\lf}^{p} (\G)$ is an intersection of a decreasing sequence of Banach algebras, hence is a Fréchet algebra.
\end{definition}

Renault's p-j-map extends to an injective continuous algebra homomorphism $S_{\lf}^{p}(\G) \hookrightarrow F_{\lambda}^{p}(\G)$, so we may view $S_{\lf}^{p}(\G)$ as a subalgebra of $F_{\lambda}^{p}(\G)$ endowed with a finer Fréchet algebra topology than the one inherited from $F_{\lambda}^{p}(\G)$.

\begin{proposition} \label{prop: Fréchet-algebra is stable under holomorphic functional calculus}
	If $\G$ has strong subexponential growth with respect to $\lf$, then $S_{\lf}^{p}(\G)$ is stable under holomorphic functional calculus in $F_{\lambda}^{p}(\G)$. 
\end{proposition}
\begin{proof}
	Suppose $(\alpha_0 , \beta_0)$ is the pair as in the definition of strong subexponential growth in \cref{def: (strongly) subexponential growth}.
	Since $S_{\lf}^{p}(\G)$ is a Fréchet algebra under a finer topology than the one inherited from $\Fp(\G)$, it suffices by \cite[Lemma 1.2]{Schweitzer:AShortProofThatMnAIsLocal} to see that $S_{\lf}^{p}(\G)$ is spectrally invariant in $F_{\lambda}^{p}(\G)$. For that, let $a \in S_{\lf}^{p}(\G)$ with inverse $a^{-1} \in F_{\lambda}^{p}(\G)$. Then $a \in \Sab(\G)$ for every $(\alpha , \beta) \in J(\alpha_0 , \beta_0)$, and is invertible in $\Fp(\G)$. By \cref{prop: weighted Lp spaces are spectral invariant in reduced groupoid Lp}, $a^{-1} \in \Sab(\G)$ for all $(\alpha , \beta) \in J(\alpha_0 , \beta_0)$, which means that $a^{-1} \in S_{\lf}^{p}(\G)$.
\end{proof}

Before we prove our main result, let us recall the definition of the K-groups for a Fréchet algebra. 
\begin{definition} \label{def: definition of the K-groups for Fréchet algebras}
	Let $\A$ be a unital Fréchet algebra. We define $K_0 (\A)$ as the Grothendieck group of the Abelian semigroup of isomorphism classes of finitely projective $\A$-modules with direct sum as the semigroup multiplication. Using the embeddings $u \mapsto \text{diag}(u,1) $, we define $K_1 (\A) = \varinjlim GL_n (\A)/GL_n (\A)_0 $, where the $GL_n (\A)$ are the invertible matrices in the Fréchet algebra $M_n (\A)$ endowed with the induced topology, and $GL_n (\A)_0$ is the normal subgroup given by the path component of the identity. If $\A$ is not unital, $K_{\ast}(\A)$ is defined to be the kernel of the naturally induced map from $K_{\ast}(\tilde{\A})$ to $K_{\ast}(\C)$.
\end{definition}
When $\A$ is a Banach algebra, the above defines its usual K-groups.

\begin{theorem} \label{thm: K-groups are independent of p}
	If $\G$ has strong subexponential growth with respect to some locally bounded length function, then $K_\ast (F_{\lambda}^{p}(\G))$, for $\ast = 0,1$, is independent of $p \in [1, \infty)$.
\end{theorem}
\begin{proof}
	Suppose $(\alpha_0 , \beta_0) $ are the constants as in the definition of strong subexponential growth for $\G$ in \cref{def: (strongly) subexponential growth}, and let $S_{\lf}^{p}(\G)$, for $p \in [1, \infty)$, be the associated Fréchet algebras of strongly subexponentially decreasing functions. Combining \cref{prop: Fréchet-algebra is stable under holomorphic functional calculus} with \cite[Lemma 2.3]{AustadOrtegaPalmstrom:PolynomialGrowthAndPropertyRDpForEtaleGroupoids}, we have that $$ K_\ast (S_{\lf}^{p}(\G)) \cong K_\ast (F_{\lambda}^{p}(\G)) ,$$ for $\ast = 0,1$ and $p \in [1, \infty)$. Therefore, it suffices to prove that $$ S_{\lf}^{p}(\G) = S_{\lf}^{1}(\G) ,$$ as Fréchet algebras, for all $p \in (1, \infty)$. It is straightforward to verify that $$ \| a \|_{\alpha, \beta, p} \leq \| a \|_{\alpha, \beta, 1} ,$$ for all $a \in S_{\lf, (\alpha, \beta)}^{1}(\G)$ and $(\alpha, \beta) \in J(\alpha_0 , \beta_0)$, where $\| \cdot \|_{\alpha, \beta, p}$ is the norm on $\Sab(\G)$. 
	%so 
	It follows that the inclusion $S_{\lf}^{1}(\G) \subset S_{\lf}^{p}(\G)$ is continuous. Conversely, fix $(\alpha, \beta) \in J(\alpha_0 , \beta_0)$. By an application of Hölder's inequality together with \cref{lem: subexponential growth gives norm estimate of exponential} it follows that there exists a constant $K_{\alpha, \beta} > 0$ such that $$ \| a \|_{\alpha, \beta, 1} \leq K_{\alpha , \beta} \| a \|_{2 \alpha, \beta, p} ,$$ for all $a \in S_{\lf, (\alpha, \beta)}^{p}(\G)$. Since also $(2 \alpha , \beta) \in J(\alpha_0 , \beta_0)$, it follows that the inclusion $S_{\lf}^{p}(\G) \subset S_{\lf}^{1}(\G)$ is continuous, so that $S_{\lf}^{p}(\G) = S_{\lf}^{1}(\G)$ as Fréchet algebras.
\end{proof}

A symmetrized version of $F_{\lambda}^{p}(\G)$, the Banach $^*$-algebra $F_{\lambda}^{p, \ast}(\G)$, has recently appeared in the literature (see for example \cite{AustadOrtega:GroupoidsAndHermitianBanachstarAlgebras, ElkiearPooya:PropertyTForBanachAlgebras, LiaoYu:KTheoryOfGroupBanachAlgebrasAndRD,PhillipsSimplicity2019,SameiWiersma:ExoticCstarAlgebrasOfGeometricGroups, SameiWiersma:QuasiHermitianLocallyCompactGroupsAreAmenable}), often under the name symmetrized $p$-pseudofunctions. In what follows, we shall recall the definition of these and explain how \cref{thm: K-groups are independent of p} can be extended to these Banach $^*$-algebras. We fix $p \in [1, \infty)$, and on $C_c (\G)$ we define the norm $$ \| f \|_{p, \ast} := \max \{ \| f \|_{\Fp} \, , \, \| f^\ast \|_{\Fp} \} .$$ 
We let $F_{\lambda}^{p, \ast}(\G)$ be completion of $C_c (\G)$ in this norm, which is a Banach $^*$-algebra. Notice that $$F_{\lambda}^{1, \ast}(\G) = L^I (\G) = \overline{C_c (\G)}^{\| \cdot \|_I} .$$
By the same argument as in 
\cite[Proposition 3.9]{AustadOrtega:GroupoidsAndHermitianBanachstarAlgebras}, the inclusion $C_c (\G) \subset \Cred(\G)$ extends to an injective contraction $F_{\lambda}^{p,\ast}(\G) \hookrightarrow \Cred(\G)$. Consequently, Renault's j-map restricts to a contractive injection $j_{p,\ast} \colon F_{\lambda}^{p,\ast}(\G) \to C_0 (\G)$. Suppose now that $\G$ has strong subexponential growth with respect to a locally bounded length function $\lf$ and constants $(\alpha_0 , \beta_0)$ as in \cref{def: (strongly) subexponential growth}. Then, since the analogous inequality as in \cref{prop: norm inequality for reduced Lp norm and application to inverse of Renaults j-p-map} holds for the norm on $F_{\lambda}^{p, \ast}(\G)$, we can argue as therein to obtain that for $(\alpha , \beta) \in J(\alpha_0 , \beta_0)$, the inverse of the restricted Renault's j-map is an injective continuous homomorphism $S_{\lf, (\alpha, \beta)}^{2}(\G) \hookrightarrow F_{\lambda}^{p, \ast}(\G)$, and hence $S_{\lf, (\alpha, \beta)}^{2}(\G)$ may be identified as a Banach subalgebra of $F_{\lambda}^{p, \ast}(\G)$ under a finer topology. 

\begin{corollary} \label{cor: Lp Banach *-algebras have same K-theory as C*-algebra}
	If $\G$ has strong subexponential growth with respect to some locally bounded length function, then $$K_\ast (F_{\lambda}^{p,\ast}(\G)) \cong K_\ast (F_{\lambda}^{p}(\G)) \cong K_\ast (C_{r}^{\ast}(\G)) ,$$ for $\ast = 0,1$ and $p \in [1, \infty)$. 
\end{corollary}
\begin{proof}
	Suppose $(\alpha_0 , \beta_0)$ is the pair of constants as in \cref{def: (strongly) subexponential growth} for $\G$. By \cref{prop: weighted Lp spaces are spectral invariant in reduced groupoid Lp}, $S_{\lf, (\alpha, \beta)}^{2} (\G)$ is spectrally invariant in $C_{r}^{\ast} (\G)$, for any $(\alpha, \beta) \in J(\alpha_0 , \beta_0)$. It then follows from the inclusions $$S_{\lf, (\alpha, \beta)}^{2} (\G) \subset F_{\lambda}^{p,\ast}(\G) \subset C_{r}^{\ast} (\G) ,$$ that $S_{\lf, (\alpha, \beta)}^{2} (\G)$ is spectrally invariant in $F_{\lambda}^{p,\ast}(\G)$ as well. Therefore, $$K_\ast (F_{\lambda}^{p, \ast}(\G)) \cong K_\ast (S_{\lf, (\alpha, \beta)}^{2} (\G)) \cong K_\ast (C_{r}^{\ast}(\G)) ,$$ for all $p \in [1, \infty)$. The result then follows by combining this with \cref{thm: K-groups are independent of p}.
\end{proof}

\section{Examples} \label{sec: examples}

The purpose of this final section is to exhibit several examples of étale groupoids with (strong) subexponential growth. In \cref{subsec: proper groupoids}, we show that all second countable proper étale groupoids are in fact of polynomial growth. Then, in \cref{subsec: coarse groupoids}, we show how the results from \cite[Section 5.1]{AustadOrtegaPalmstrom:PolynomialGrowthAndPropertyRDpForEtaleGroupoids} apply to give coarse groupoids with (strong) subexponential growth associated with many already existing examples in the literature of metric spaces. % with this growth type. 
Finally, in \cref{subsec: shift groupoid} we use ideas from \cite{Nathanson:NumberTheoryAndSemigroupsOfIntermediateGrowth} as well as recent results of Brix, Hume and Li in \cite{BrixHumeLi:MinimalCovers} to construct an étale groupoid from a shift space of infinite type which has strong subexponential growth and not polynomial.

\subsection{Proper groupoids} \label{subsec: proper groupoids}

Recall that an étale groupoid $\G$ is called \emph{proper} if the map $(r,s) \colon \G \to \G^{(0)} \times \G^{(0)} \, , \, x \mapsto (r(x) , s(x))$ is proper, that is, $(r,s)^{-1}(K) \subset \G$ is compact whenever $K \subset \G^{(0)} \times \G^{(0)}$ is compact. 

\begin{proposition} \label{prop: proper étale groupoids are of linear growth}
	If $\G$ is a second countable proper étale groupoid, then there is a locally bounded length function with respect to which $\G$ has polynomial growth.
	\begin{proof}
		Since $\G^{(0)}$ is locally compact Hausdorff and second countable, we can find an increasing sequence of compact sets $\{K_i\}_{i = 1}^{\infty}$ such that $$\G^{(0)} = \bigcup_{i = 1}^{\infty} K_i = \bigcup_{i = 1}^{\infty} {K_i}^{\circ}.$$ For each $i \in \N$, let $$\G(i) := \G(K_i) = \{x \in \G \colon s(x) , r(x) \in K_i \} ,$$ so that $\G(i) \subset \G(i+1)$, for each $i \in \N$. Then since $\G$ is proper, it follows that each $\G(i)$ is a compact subgroupoid of $\G$, and, moreover, $$ \G = \bigcup_{i = 1}^{\infty} \G(i) = \bigcup_{i = 1}^{\infty} {\G(i)}^{\circ} .$$ Let $f \colon \G \to \C$ be a positive continuous compactly supported function such that $1_{\G(i)} \leq f$. Then since the counting measures form a continuous Haar system, the map $$ u \mapsto \left| \G(i)_u \right| = \sum_{x \in \G_u} 1_{\G(i)} (x) \leq \sum_{x \in \G_u} f(x) ,$$ has a finite supremum. Let us denote this supremum by $p(i)$, so that $$ p(i) := \sup_{u \in \G^{(0)}} | \G(i)_u | < \infty .$$ Define a length function $\lf \colon \G \to \R_{\geq 0}$ as follows: $\lf(u) := 0$, for all $u \in \G^{(0)}$, and if $x \in \G \setminus \G^{(0)} $ then $$\lf(x) := \min \{ p(i) \colon x \in \G(i) \}.$$ We claim that $\lf$ is a locally bounded length function. Indeed, given $x \in \G$, there exists $i \in \N$ such that $x \in \G(i)$, and for any such $i$, $x \in \G(i)$ if and only if $x^{-1} \in \G(i)$, and therefore $\lf(x^{-1}) = \lf(x)$, for all $x \in \G$. Moreover, if $(x, y) \in \G^{(2)}$, say $x \in \G(n)$ and $y \in \G(m)$, let $k := \max \{ n,m \}$. Then $x, y , x y \in \G(k)$, so that $$ \lf(x y) \leq p(k) \leq p(n) + p(m) ,$$ for any such $n$ and $m$. Thus $\lf(x y) \leq \lf(x) + \lf(y)$, so that $\lf$ is indeed a length function. To see that it is locally bounded, let $K \subset \G$ be compact. Since also $\G = \bigcup_{i = 1}^{\infty} {\G(i)}^{\circ}$, compactness produces an $N \in \N$ such that $K \subset \G(N)$, and so $\lf(K) \leq p(N) < \infty$, so $\lf$ is a locally bounded length function. Now let $N \in \N$ be given, and if possible, find the largest $k \in \N$ such that $p(k) \leq N$. By definition, if $x \notin \G(k)$, then $\lf(x) > N$, and
		so, for any $u \in \G^{(0)}$, $$ \left| B_{\G_u} (N) \right| = \left| \{ x \in \G_u \colon \lf(x) \leq N \} \right| \leq | \{ x \in \G(k)_u \} | \leq p(k) \leq N .$$ If no such $k \in \N$ exist, then $\left| B_{\G_u} (N) \right| = 1 \leq N$, for any $u \in \G^{(0)}$. Therefore, $\G$ has polynomial growth. 
	\end{proof}
\end{proposition}
 
\begin{remark} \label{rmk: Paravicini's results on proper groupoids gives also the same result}
	Suppose $\G$ is a second countable proper étale groupoid. By \cref{prop: proper étale groupoids are of linear growth}, $\G$ has polynomial growth. Hence for any %$\alpha \in \R_{> 0}$ 
	$\alpha > 0$ and $\beta \in (0,1)$, $S_{\lf, (\alpha, \beta)}^{2} (\G)$ is a Banach algebra that is an unconditional completion of $C_c (\G)$ which is continuously contained in $C^*_r(\G)$. Therefore, combining \cref{cor: weighted Lp spaces all capture K-theory} and \cref{thm: K-groups are independent of p} with  \cite[Proposition 2.4]{Paravicini:TheBostConjectureAndProperBanachAlgebras}, we have that $$K_\ast (F_{\lambda}^{p}(\G)) \cong K_\ast (\mathcal{A}(\G)) ,$$
	for any $p\in [1,\infty)$ and any unconditional completion $\mathcal{A}(\G)$ of $C_c (\G)$ such that the inclusion $\mathcal{A}(\G)\subseteq C^*_r(\G)$ is continuous. 
\end{remark}

\subsection{Coarse groupoids} \label{subsec: coarse groupoids}

Let us briefly recall the definition of coarse groupoids: We say that an extended metric space $(X,d)$ is \emph{uniformly locally finite} if for every $R >0$ we have
\begin{align*}
	\sup_{x \in X} \vert \overline{B}(x,R) \vert < \infty 
\end{align*}
where $\overline{B}(x,R)$ is the closed ball of radius $R$ around $x$.  %One associates 
From this we construct an étale groupoid, denoted by $\G_{(X,d)}$, as follows. For every $r  > 0$, let $E_r := \{ (x,y) \in X \times X \mid d(x,y) \leq r \}$, and define 
\begin{align*}
	\G_{(X,d)} := \bigcup_{r \geq 0} \overline{E_r}
\end{align*}
where the closure is taken in the Stone-\v{C}ech compactification $\beta X \times \beta X$. The unit space $\G_{(X,d)}^{(0)}$ is identified with $\overline{E_0} \cong \beta X$. The range and source maps are the natural extensions of the first and second projection map on $X\times X$ to the Stone-\v{C}ech compactification, and the multiplication map is inherited from the pair groupoid multiplication on $\beta X \times \beta X$. 

By \cite[Lemma 5.6]{AustadOrtegaPalmstrom:PolynomialGrowthAndPropertyRDpForEtaleGroupoids} we know that the metric $d$ naturally extends to a metric $\beta d$ on $\G_{(X,d)}$ which induces a continuous and proper length function $\lf_{\beta d}$ on $\G_{(X,d)}$. We recall the following result from \cite[Proposition 5.8]{AustadOrtegaPalmstrom:PolynomialGrowthAndPropertyRDpForEtaleGroupoids}.
\begin{proposition}\label{prop:growth-in-Stone-Cech}
	Let $(X,d)$ be a uniformly locally finite extended metric space, and suppose there is a function $f \colon \R_{\geq 0} \to \R_{\geq 0}$ for which $\vert \overline{B}(x,r)\vert \leq f(r)$ for all $x \in X$ and all $r \geq 0$. Let $\chi \in \beta X = \G_{(X,d)}^{(0)}$. Then
	\begin{align*}
		\vert B_{(\G_{(X,d)})_\chi} (r) \vert \leq 
		\begin{cases}
			f(r) & \chi \in X \\
			M f(r)^2 & \chi \in \beta X \setminus X
		\end{cases}
	\end{align*}
	for a sufficiently large constant $M$. That is, the growth of the groupoid $\G_{(X,d)}$ is bounded above by the growth of $f^2$. 
\end{proposition}

We immediately obtain the following corollary.
\begin{corollary}\label{cor:growth-coarse-groupoids}
	Let $(X,d)$ be an extended metric space for which $\vert \overline{B}(x,r)\vert \leq f(r)$ for all $x \in X$ and all $r \geq 0$. If there are $C, \alpha > 0$ and $0 < \beta < 1$ for which $f(r) \leq C \exp (\alpha r^\beta)$ for all $r \geq 0$, then $\G_{(X,d)}$ has strong subexponential growth with respect to $\lf_{\beta d}$. If $\lim_{r \to \infty} f(r)^{1/r} = 1$, then $\G_{(X,d)}$ has subexponential growth with respect to $\lf_{\beta d}$. 
\end{corollary}
\begin{proof}
	By the assumptions on $f$ we have  $\vert \overline{B}(x,r)\vert \leq C \exp(\alpha r^\beta)$. By \cref{prop:growth-in-Stone-Cech} there exists $M >0$ such that we can guarantee 
	\begin{align*}
		\vert (\G_{(X,d)})_u \vert \leq MC \exp(2\alpha r^\beta)
	\end{align*}
	uniformly in $u$, from which we deduce that $\G_{(X,d)}$ has strong subexponential growth with respect to $\lf_{\beta d}$.
	
	The statement for subexponential growth follows similarly.
\end{proof}

There are several interesting metric spaces with (strong) subexponential growth coming from graph theory, see for example \cite{BondarenkoEtAl2012, KontogeorgiouWinter2022, Lehner2016, MiasnikovSavchuk2015}. By the argument in the proof of \cref{cor:growth-coarse-groupoids}, we deduce that for any one of these examples, the associated coarse groupoid has also (strong) subexponential growth. Moreover, if the original extended metric space has (strong) subexponential growth but not polynomial growth, then the same is true for the associated coarse groupoid. This follows by observing that for any point $x \in X \subset \beta X = \G_{(X,d)}^{(0)}$, we have $\vert B_{(\G_{(X,d)})_x} (r) \vert = \vert \overline{B}(x,r)\vert$.

\subsection{Shift groupoids} \label{subsec: shift groupoid}
Let us start by recalling the definition of a shift space and its associated Renault-Deaconu groupoid.
Fix a finite set $\A$ that we call an \emph{alphabet}. A \emph{path} is a map $x:\mathbb{N}\to \A$, and we denote by $\A^{\mathbb{N}}$ the set of all paths. Given $x \in \A^{\mathbb{N}}$, and $w=(w_1,\cdots,w_n)\in \A^n$, we define the concatenation of a word $w \in \A^n$ with a path $x \in \A^{\N}$ to be the path 
$$ax[i]=\left\lbrace \begin{array}{ll} w_i & \text{if }1\leq i\leq n \\ x[i-n] & \text{ if }i\geq n+1
\end{array}\right. $$ Whenever $x \in \A^{\N}$, by $x[n,n+k]$ we mean the word $(x[n], x[n+1] , \ldots , x[n+k]) \in \A^{k+1}$, for $n \in \N$ and $k \in \N_0$.
We give $\A^{\mathbb{N}}$ the product topology, so that $\A^{\mathbb{N}}$ is a Cantor space, with topology generated by the sets of the form $$Z(w)=\{x \in \A^{\mathbb{N}}: x[i]=w_i \text{ for }1\leq i\leq n\} ,$$ for $w=(w_1,\cdots, w_n)\in \A^n$. 
Define the \emph{shift map}
$\sigma:\A^{\mathbb{N}}\to \A^{\mathbb{N}}$  by $\sigma(x)[n]=x[n+1]$ for every $n\in \mathbb{N}$. A (one-sided) \emph{shift space} is a closed subset $X\subseteq \A^{\mathbb{N}}$ such that $\sigma(X)\subseteq X$.
A \emph{word of length $n$ in }$X$ is an element $w=(w_1,\cdots,w_n)\in \A^n$ such that there exists $x\in X$ and $k\in \mathbb{N}$ such that $x[k+i-1]=w_i$ for $1\leq i\leq n$. We denote by $\mathcal{L}_n(X)$ the set of all words of length $n$ in $X$, and $\mathcal{L}_*(X)=\bigcup_{n\in \mathbb{N}}\mathcal{L}_n(X)$ the set of all words of $X$. We define the \emph{complexity function of} $X$ to be the function $p_X(n)=\sum_{i=1}^n  |\mathcal{L}_i(X)|$. 
The \emph{Renault-Deaconu groupoid} associated with the shift space $\sigma:X\to X$ is the groupoid
$$\mathcal{G}_X:=\{(x,m-n,y)\in X\times\mathbb{Z}\times X: \sigma^m(x)=\sigma^n(y)\}\,.$$
with multiplication $(x,m-n,y) \cdot (y,k-l,z) = (x, m+k - (n+l) , z)$, inversion $(x,m-n,y)^{-1} = (y,n-m,x)$, and unit space $\mathcal{G}^{(0)}_X = \{ (x,0,x) \colon x \in X \}$ identified with $X$. 
The topology has as basis sets of the form 
$$Z(u,v):=\{(x,m-n,y)\in \mathcal{G}_X: \sigma^m(x)=\sigma^n(y) \text{ for }x\in Z(u)\text{ and }y\in Z(v)\}  ,$$ for $u\in \mathcal{L}_m(X)$ and $v\in \mathcal{L}_n(X)$
We allow $u$ and $v$ to be the empty word, in which case 
$$Z(u,\emptyset):=\{(x,m,y)\in \mathcal{G}_X: \sigma^m(x)=y \text{ for }x\in Z(u)\text{ and }y\in X\} ,$$ for $u\in \mathcal{L}_m(X)$ 
and 
$$Z(\emptyset,v):=\{(x,-n,y)\in \mathcal{G}_X: \sigma^n(y)=x \text{ for }x\in X\text{ and }y\in Z(v)\} ,$$ for $v\in \mathcal{L}_n(X)$.
The shift map $\sigma$ is always locally injective. It is a local homeomorphism if and only if $X$ is a shift of \emph{finite type}, meaning that there  exists $K \in \N$ and a subset $\mathcal{F}\subseteq \mathcal{A}^K$, called the forbidden words, such that for every $x\in X$ and $n\in \N$ we have that $x[n,n+K-1]\notin\mathcal{F}$. In this case, observe that $\mathcal{F}=\mathcal{A}^K\setminus \mathcal{L}_K(X)$. A shift is of \emph{infinite type} if it is not of finite type. Given a shift space $X$, the associated Renault-Deaconu groupoid $\mathcal{G}_X$ is étale if $X$ is a shift of finite type, and r-discrete if $X$ is a shift of infinite type.

Now, observe that the set 
\begin{equation} \label{eq: generating set for shift groupoids}
	S:=\bigcup_{a\in \A} Z(a,\emptyset) \cup  Z(\emptyset,a)
\end{equation} 
is a generating set for $\mathcal{G}_X$. Indeed, given $g=(x,m-n,y)\in \mathcal{G}_X$, we have that 
$$(x,m-n,y)=(x,1,\sigma (x))(\sigma (x),1,\sigma^2 (x))\cdots (\sigma^{m-1}(x),1,\sigma^m (x))(\sigma^n (y),-1,\sigma^{n-1} (y))\cdots (\sigma (y),-1,y) ,$$
and therefore,
$$g\in Z(x[1],\emptyset)\cdots Z(x[m],\emptyset)Z(\emptyset,y[n])\cdots Z(\emptyset,y[1])\,.$$
We denote by $\lf_S:\mathcal{G}_X\to \mathbb{N}\cup \{0\}$ the associated length function, which we recall is given by $\lf (x) = 0$ for $x \in X$, and $$ \lf (g) = \inf\{k \in \N \colon g \in S^k \} ,$$ when $g \in \G_X \setminus X$.

Now let $x \in X$ and define $W_x (n) :=\{g\in (\mathcal{G}_X)_x: \lf_S (g)=n\}$, for $n \in \N_0$. First observe that 
$$W_x(1)=\{(ax,1,x): a\in \mathcal{L}_1(X)\text{ such that }ax\in X\}\cup \{(\sigma (x),-1,x)\}\,,$$
therefore $|W_x(1)|\leq p_X(1)+1$. Moreover,

\begin{align*}W_x(2) & =\{(wx,2,x): w\in \mathcal{L}_2(X)\text{ such that }wx\in X\}\cup \\
	& \cup \{(a\sigma (x),0,x):a\in \mathcal{L}_1(X)\text{ such that }a\sigma (x)\in X\}\cup \{(\sigma^2 (x),-2,x)\}\,,\end{align*}
and therefore we have that $|W_x(2)|\leq p_X(2)+p_X(1)+1$, and in general,
$$|W_x(n)|\leq 1+ \sum_{k=1}^n p_{X}(k)\leq 1+np_X(n)\,,$$ for $n \in \N$, so that 
$$ |B_{(\G_X)_x} (n)| = 1 + \sum_{k = 1}^{n} |W_x (k)| \leq 1 + n + n^2 p_X (n) .$$

It follows immediately from this that if the complexity function $p_X$ has (strong) subexponential growth, then the function $n \mapsto \sup_{x \in X} |B_{(\G_X)_x} (n)|$ also has (strong) subexponential growth. Let us record this result for future reference.

\begin{lemma} \label{lem: subexponential complexity function gives subexponential growth groupoid}
	Let $X$ be a shift and $\G_X$ the associated Renault-Deaconu groupoid. If the complexity function has (strong) subexponential growth, then the groupoid $\G_X$ has (strong) subexponential growth. 
\end{lemma}

If $X$ is a shift of finite type, then it is a standard result that $X$ is conjugate to a shift space associated with an infinite path space of a finite directed graph \cite[Section 2.2]{LindMarcus:SymbolicDynamics}.  Then combining this with \cite[Section 5.3]{AustadOrtegaPalmstrom:PolynomialGrowthAndPropertyRDpForEtaleGroupoids} we have the following. 

\begin{proposition} \label{prop: groupoids from shifts of finite type are graph groupoids and hence have the same dichotomy of growth}
	If $X$ is a shift of finite type, then there is a finite directed graph $E_X$ such that $\G_{E_X} \cong \G_X$. Consequently, with respect to the natural length function, the Renault-Deaconu groupoid $\G_{X}$ is either of exponential or polynomial growth.
	   
\end{proposition}

Since we are interested in exhibiting examples of groupoids associated with shifts whose growth is strongly subexponential and not polynomial, we see from \cref{prop: groupoids from shifts of finite type are graph groupoids and hence have the same dichotomy of growth} that it is natural to look for examples among the shifts of \emph{infinite} type. The shift that will turn out to give our desired example is one that we call the \emph{ordered prime shift}, which we define next: The alphabet is $\mathcal{A} = \{0,1\}$, and we let 
$$\mathcal{F} := \left\{ 1 0^n 1 \colon n \in \N_0 \text{ is not a prime} \right\} \cup \left\{ 10^{p_1}10^{p_2}1 \ldots 10^{p_r}1 \colon r \geq 2 \, \, p_i \geq p_{i+1} \text{ for some } i \right\} ,$$ be the \emph{forbidden words}. The associated shift $X := X_{\mathcal{F}}$ is the space of all paths consisting of \emph{admissible words}, that is, the paths $x$ such that for any $n \in \N$ and $k \in \N_0$, the word $x[k,n+k]$ is not a forbidden word. Notice that any admissible word has one of three forms: 

\begin{equation} \label{eq: only zeros}
	0^k \text{ for } k \in \N_0 ;
\end{equation}

\begin{equation} \label{eq: one in the middle}
	0^k 1 0^l \text{ for } l,k \in \N_0 ;
\end{equation}

\begin{equation} \label{eq: primes in the middle}
	0^k 1 0^{p_1} 1 0^{p_2} 1 \ldots 0^{p_r} 1 0^l \text{ for } k,l \in \N_0 \, , \, r \in \N \, , p_i \text{ prime } \forall i \text{ and } p_i < p_{i+1} \, \, \forall i.
\end{equation}

\begin{lemma} \label{lem: complexity function of prime shift is of intermediate growth}
	The complexity function $p_X$ of the ordered prime shift $X$ has strong subexponential growth, but not polynomial growth.
\end{lemma}
There are two key tools to proving the above. First, Chebyshev's theorem which gives a useful estimate for the map $x \mapsto \pi(x)$ that counts the number of prime numbers not exceeding $x \in \R_{\geq 2}$ (see \cite[Theorem 6.3]{Nathanson:AdditivNumberTheory}): there exists constants $c_1 , c_2 > 0$ such that for all $x \in \R_{\geq 2}$,
\begin{equation} \label{eq: Chebyshevs inequality for primes}
	\frac{c_1 x}{\ln(x)} \leq \pi(x) \leq \frac{c_2 x}{\ln (x)}.
\end{equation}
Second, the following estimate for the Hardy-Ramanujan partition function (see \cite{HardyRamanujan:AsymptoticFormulaeInCombinatoryAnalysis}): there exists $A,B > 0$ such that for all sufficiently large $n$, 
\begin{equation} \label{eq: Hardy-Ramanujan partition bounds}
	e^{A \sqrt{n}} \leq p(n) \leq e^{B \sqrt{n}} ,
\end{equation}
where we recall that $p(n)$ counts the number of unrestricted partitions of the positive integer $n$.

\begin{proof} [Proof of \cref{lem: complexity function of prime shift is of intermediate growth}]
	First we show that $p_X$ cannot have polynomial growth. Consider the words of the form $$10^{p_1} 1 0^{p_2} 1 \ldots 0^{p_r} 1,$$ where $p_i$ are all distinct primes in increasing order which does not exceed $\sqrt{n/2}$. Such a word has length $$ r + 1 + \sum_{i = 1}^{r} p_i \leq r+1 + r\sqrt{n/2} \leq n/2 + \sqrt{n/2} \pi(\sqrt{n/2}) \leq n/2 + \frac{2 c_2 n/2}{\log(n/2)} \leq n ,$$ for all large $n$, where we have used \cref{eq: Chebyshevs inequality for primes}. Let $\{p_1 , \ldots , p_{\pi(\sqrt{n/2})}\}$ be the collection of primes not exceeding $\sqrt{n/2}$. As we saw above, any sub-collection gives a unique word as above of length less than or equal to $n$. Thus, using \cref{eq: Chebyshevs inequality for primes}, we have $$ p_X (n) \geq \# \text{ subsets of } \{p_1 , \ldots , p_{\pi(\sqrt{n/2})}\} = 2^{\pi(\sqrt{n/2})} \geq 2^{\frac{2c_1 \sqrt{n/2}}{\log(n/2)}} ,$$ so that $p_X $ cannot be dominated by any polynomial. 
	
	Next we show that $p_X$ has strong subexponential growth. For this we will use some rough estimates on the number of elements of length at most $n$. Recall that any word has one of the three forms \cref{eq: only zeros}, \cref{eq: one in the middle} and \cref{eq: primes in the middle}. There are $n+1$ words of the form \cref{eq: only zeros} with length at most $n$, and there are at most $n^2$ words of the form \cref{eq: one in the middle} with length at most $n$. Let us then consider the words of the form \cref{eq: primes in the middle}. Fix any $1 \leq k \leq n$ and consider the set of all primes $\{p_1 , \ldots , p_{\pi(k)} \}$ not exceeding $k$. A sub-collection $\{p_{i_1} , \ldots , p_{i_r}\}$ that is arranged in increasing order gives a unique element $1 0^{p_{i_{1}}} 1 0^{p_{i_{2}}} 1 \ldots 1 0^{p_{i_{r}}} 1$ of length possibly less than $k$. 
	There are $p(k)$ unrestricted partitions of $k$ into a sum of non-negative integers, and hence there are at most $p(k - (r+1)) \leq p(k)$ elements of the form $1 0^{p_{i_1}} 1 0^{p_{i_2}} 1 \ldots 1 0^{p_{i_r}} 1$ such that $\sum_j p_{i_j} + r + 1 = k$. For every such element $1 0^{p_{i_1}} 1 0^{p_{i_2}} 1 \ldots 1 0^{p_{i_r}} 1$, there are at most $n^2$ elements of the form $0^t 1 0^{p_{i_1}} 1 0^{p_{i_2}} 1 \ldots 1 0^{p_{i_r}} 1 0^l$ with $0 \leq t+l \leq n-k$. Therefore, the number of elements of the form \cref{eq: primes in the middle} with length at most $n$ is bounded by $$\sum_{k = 1}^{n} n^2 p(k) \leq n^3 p(n) .$$ This together with \cref{eq: Hardy-Ramanujan partition bounds} implies that there exist constants $C , \alpha > 0$ large enough such that $$ p_X (n) \leq n+1 + n^2 + n^3 p(n) \leq 4n^3 p(n) \leq C e^{\alpha \sqrt{n}} , $$ for all $n \in \N$. This proves the lemma.
\end{proof}

\begin{proposition} \label{prop: groupoid associated to the prime shift is of subexponential growth}
	With respect to its canonical length function, the Renault-Deaconu groupoid associated with the ordered prime shift has strong subexponential growth, and is not of polynomial growth.
	\begin{proof}
		Let $\G_X$ denote the groupoid associated with the ordered prime shift $X$. Combining \cref{lem: subexponential complexity function gives subexponential growth groupoid} with \cref{lem: complexity function of prime shift is of intermediate growth}, we see that $\G_X$ has strong subexponential growth. Let us show next that it is not of polynomial growth. For this, let $x \in X$ be the path $x = 0^{\infty}$. Concatenating any admissible word $w$ with $x$ yields another element $wx \in X$, and it therefore follows that $$ | B_{(\G_X)_x} (n) | \geq p_X (n) .$$ Since the complexity function $p_X $ is not dominated by any polynomial, the same must be true for the function $n \mapsto | B_{(\G_X)_x} (n) |$, hence the result. 
	\end{proof}
\end{proposition}

As already remarked, the Renault-Deaconu groupoid associated with the ordered prime shift is not an étale groupoid, but an r-discrete groupoid. To obtain an étale groupoid with the same growth, we can apply recent results of Brix, Hume and Li in \cite{BrixHumeLi:MinimalCovers}. Therein, given an r-discrete groupoid $\G$, the authors associate an étale groupoid $\hat{\G}$ called the \emph{cover} of $\G$. The cover groupoid can be viewed as a transformation groupoid built from a groupoid action of $\G$ on a locally compact Hausdorff space $\hat{X}$, so $\hat{\G} = \G \ltimes \hat{X}$ (see \cite[Proposition 7.5]{BrixHumeLi:MinimalCovers}). When $\G$ is the Renault-Deaconu groupoid associated with a shift of infinite type, we endow $\hat{\G}$ with the length function $\lf \colon \hat{\G} \to \R_{\geq 0}$ given by $\lf (g,x) = \lf_S (g)$, where $\lf_S$ is the length function on $\G$ induced by the generating set $S$ as in \cref{eq: generating set for shift groupoids}. With this choice of length function, it is clear that $\hat{\G}$ and $\G$ share the same growth properties, and hence we arrive at a desired example.

\begin{theorem} \label{thm: étale groupoid built from prime shift}
	Let $\G_X$ denote the Renault-Deaconu groupoid associated with the ordered prime shift $X$, and let $\hat{\G}_X$ denote its cover groupoid. Then $\hat{\G}_X$ is an étale groupoid that has strong subexponential growth, but not polynomial growth.
\end{theorem}

\printbibliography

@article {GardellaThiel:BanachAlgebrasGeneratedByAnInvertibleIsometryOfAnLpSpace,
	AUTHOR = {Gardella, E. and Thiel, H.},
	TITLE = {Banach algebras generated by an invertible isometry of an
	{$L^p$}-space},
	JOURNAL = {J. Funct. Anal.},
	FJOURNAL = {Journal of Functional Analysis},
	VOLUME = {269},
	YEAR = {2015},
	NUMBER = {6},
	PAGES = {1796--1839},
	ISSN = {0022-1236,1096-0783},
	MRCLASS = {46J10 (22D20 37A55 47L10)},
	MRNUMBER = {3373434},
	MRREVIEWER = {H.\ G.\ Dales},
}

@article {GardellaThiel:ExtendingRepresentationsOfBanachAlgebrasToTheirBiduals,
	AUTHOR = {Gardella, E. and Thiel, H.},
	TITLE = {Extending representations of {B}anach algebras to their
	biduals},
	JOURNAL = {Math. Z.},
	FJOURNAL = {Mathematische Zeitschrift},
	VOLUME = {294},
	YEAR = {2020},
	NUMBER = {3-4},
	PAGES = {1341--1354},
	ISSN = {0025-5874,1432-1823},
	MRCLASS = {46H15 (46E30 46L05 47L10)},
	MRNUMBER = {4074042},
	MRREVIEWER = {Dinesh\ Jayantilal\ Karia},
}

@article {HetlandOrtega:RigidityOfTwistedGroupoidLpOperatorAlgebras,
	AUTHOR = {Hetland, E. V. and Ortega, E.},
	TITLE = {Rigidity of twisted groupoid {$L^p$}-operator algebras},
	JOURNAL = {J. Funct. Anal.},
	FJOURNAL = {Journal of Functional Analysis},
	VOLUME = {285},
	YEAR = {2023},
	NUMBER = {6},
	PAGES = {Paper No. 110037, 45},
	ISSN = {0022-1236,1096-0783},
	MRCLASS = {22D20 (22D25 43A15 47B01)},
	MRNUMBER = {4601857},
	MRREVIEWER = {Xiao\ Chen},
}

@article{BaradynKwasniewski:TopFreeActionsAndIdealsInTwistedBanachAlgebraCrossedProducts,
	title={Topologically free actions and ideals in twisted Banach algebra crossed products},
	author={Bardadyn, K. and Kwa{\'s}niewski, B.},
	journal={Proceedings of the Royal Society of Edinburgh Section A: Mathematics},
	pages={1--31},
	year={2023},
	publisher={Royal Society of Edinburgh Scotland Foundation}
}

@article {GardellaThiel:GroupAlgebrasActingOnLpSpaces,
	AUTHOR = {Gardella, E. and Thiel, H.},
	TITLE = {Group algebras acting on {$L^p$}-spaces},
	JOURNAL = {J. Fourier Anal. Appl.},
	FJOURNAL = {The Journal of Fourier Analysis and Applications},
	VOLUME = {21},
	YEAR = {2015},
	NUMBER = {6},
	PAGES = {1310--1343},
	ISSN = {1069-5869,1531-5851},
	MRCLASS = {22D20 (43A07 43A15 43A65 46E30 47L10)},
	MRNUMBER = {3421918},
	MRREVIEWER = {Pekka\ Salmi},
}

@article {GardellaLupini:RepresentationsOfEtaleGroupoidsOnLpSpaces,
	AUTHOR = {Gardella, E. and Lupini, M.},
	TITLE = {Representations of \'etale groupoids on {$L^p$}-spaces},
	JOURNAL = {Adv. Math.},
	FJOURNAL = {Advances in Mathematics},
	VOLUME = {318},
	YEAR = {2017},
	PAGES = {233--278},
	ISSN = {0001-8708,1090-2082},
	MRCLASS = {47L10 (22A22 46H05)},
	MRNUMBER = {3689741},
	MRREVIEWER = {Alexander\ Isaakovich\ Shtern},
}

@article {GardellaThiel:RepresentationsOfPConvolutionAlgebrasOnLqSpaces,
	AUTHOR = {Gardella, E. and Thiel, H.},
	TITLE = {Representations of {$p$}-convolution algebras on
	{$L^q$}-spaces},
	JOURNAL = {Trans. Amer. Math. Soc.},
	FJOURNAL = {Transactions of the American Mathematical Society},
	VOLUME = {371},
	YEAR = {2019},
	NUMBER = {3},
	PAGES = {2207--2236},
	ISSN = {0002-9947,1088-6850},
	MRCLASS = {47L10 (43A15 43A65 46E30)},
	MRNUMBER = {3894050},
	MRREVIEWER = {Sedigheh\ Barootkoob},
}

@article {Herz:TheTheoryOfPSpacesWithAnApplicationToConvolution,
	AUTHOR = {Herz, C.},
	TITLE = {The theory of {$p$}-spaces with an application to convolution
	operators},
	JOURNAL = {Trans. Amer. Math. Soc.},
	FJOURNAL = {Transactions of the American Mathematical Society},
	VOLUME = {154},
	YEAR = {1971},
	PAGES = {69--82},
	ISSN = {0002-9947,1088-6850},
	MRCLASS = {22.65},
	MRNUMBER = {272952},
	MRREVIEWER = {A.\ B.\ Willcox},
}

@misc{BardadynKwaśniewskiMcKee:BanachAlgebrasAssociatedToTwistedÉtaleGroupoidsSimplicityAndPureInfiniteness,
	title={Banach algebras associated to twisted \'{e}tale groupoids: simplicity and pure infiniteness}, 
	author={K. Bardadyn and B. Kwaśniewski and A. McKee},
	year={2024},
	eprint={2406.05717},
	archivePrefix={arXiv},
	primaryClass={math.FA},
	url={https://arxiv.org/abs/2406.05717}, 
}

@misc{Phillips:CrossedProductsOfLpAlgebrasAndKtheoryOfCuntzAlgebras,
	title={Crossed products of {$L^p$} operator algebras and the K-theory of Cuntz algebras on {$L^p$} spaces}, 
	author={N. C. Phillips},
	year={2013},
	eprint={1309.6406},
	archivePrefix={arXiv},
	primaryClass={math.FA},
	url={https://arxiv.org/abs/1309.6406}, 
}

@misc{Phillips:SimplicityOfUHFAndCuntzAlgebrasOnLpSpaces,
	title={Simplicity of UHF and Cuntz algebras on {$L^p$}~spaces}, 
	author={N. C. Phillips},
	year={2013},
	eprint={1309.0115},
	archivePrefix={arXiv},
	primaryClass={math.FA},
	url={https://arxiv.org/abs/1309.0115}, 
}

@article {WangWang:NoteOnTheEllpToeplitzAlgebra,
	AUTHOR = {Wang, Q. and Wang, Z.},
	TITLE = {Notes on the {$\ell^p$}-Toeplitz algebra on {$\ell^p(\mathbb{N})$}},
	JOURNAL = {Israel J. Math.},
	FJOURNAL = {Israel Journal of Mathematics},
	VOLUME = {245},
	YEAR = {2021},
	NUMBER = {1},
	PAGES = {153--163},
	ISSN = {0021-2172,1565-8511},
	MRCLASS = {46H35 (47L80)},
	MRNUMBER = {4357458},
}

@article {PooyaShirin:SimpleReducedLpOperatorCrossedProductsWithUniqueTrace,
	AUTHOR = {Hejazian, S. and Pooya, S.},
	TITLE = {Simple reduced {$L^p$}-operator crossed products with unique
	trace},
	JOURNAL = {J. Operator Theory},
	FJOURNAL = {Journal of Operator Theory},
	VOLUME = {74},
	YEAR = {2015},
	NUMBER = {1},
	PAGES = {133--147},
	ISSN = {0379-4024,1841-7744},
	MRCLASS = {47L10 (46H05 46H35)},
	MRNUMBER = {3383617},
	MRREVIEWER = {Daniel\ Gon\c calves},
}

@article {BlecherPhillips:LpOperatorAlgebrasWithApproximateIdentities,
	AUTHOR = {Blecher, D. P. and Phillips, N. C.},
	TITLE = {{$L^p$}-operator algebras with approximate identities, {I}},
	JOURNAL = {Pacific J. Math.},
	FJOURNAL = {Pacific Journal of Mathematics},
	VOLUME = {303},
	YEAR = {2019},
	NUMBER = {2},
	PAGES = {401--457},
	ISSN = {0030-8730,1945-5844},
	MRCLASS = {47L30 (46H10 46H35 47B38 47L10)},
	MRNUMBER = {4059950},
	MRREVIEWER = {Sedigheh\ Barootkoob},
}

@misc{Ma:FiberwiseAmenabilityOfAmpleGroupoids,
	title={Fiberwise amenability of ample \'{e}tale groupoids}, 
	author={X. Ma},
	year={2021},
	eprint={2110.11548},
	archivePrefix={arXiv},
	primaryClass={math.OA},
	url={https://arxiv.org/abs/2110.11548}, 
}

@misc{LiaoYu:KTheoryOfGroupBanachAlgebrasAndRD,
	title={K-theory of group Banach algebras and Banach property RD}, 
	author={B. Liao and G. Yu},
	year={2017},
	eprint={1708.01982},
	archivePrefix={arXiv},
	primaryClass={math.FA},
	url={https://arxiv.org/abs/1708.01982}, 
}

@article {Nathanson:NumberTheoryAndSemigroupsOfIntermediateGrowth,
	AUTHOR = {Nathanson, M. B.},
	TITLE = {Number theory and semigroups of intermediate growth},
	JOURNAL = {Amer. Math. Monthly},
	FJOURNAL = {American Mathematical Monthly},
	VOLUME = {106},
	YEAR = {1999},
	NUMBER = {7},
	PAGES = {666--669},
	ISSN = {0002-9890,1930-0972},
	MRCLASS = {20M05},
	MRNUMBER = {1720447},
}

@book {Nathanson:AdditivNumberTheory,
	AUTHOR = {Nathanson, M. B.},
	TITLE = {Additive number theory},
	SERIES = {Graduate Texts in Mathematics},
	VOLUME = {165},
	NOTE = {Inverse problems and the geometry of sumsets},
	PUBLISHER = {Springer-Verlag, New York},
	YEAR = {1996},
	PAGES = {xiv+293},
	ISBN = {0-387-94655-1},
	MRCLASS = {11Bxx (11-02)},
	MRNUMBER = {1477155},
	MRREVIEWER = {Yuri\ Bilu},
}

@incollection {HardyRamanujan:AsymptoticFormulaeInCombinatoryAnalysis,
	AUTHOR = {Hardy, G. H. and Ramanujan, S.},
	TITLE = {Asymptotic formul\ae\ in combinatory analysis [{P}roc.
	{L}ondon {M}ath. {S}oc. (2) {\bf 16} (1917), {R}ecords for 1
	{M}arch 1917]},
	BOOKTITLE = {Collected papers of {S}rinivasa {R}amanujan},
	PAGES = {244},
	PUBLISHER = {AMS Chelsea Publ., Providence, RI},
	YEAR = {2000},
	ISBN = {0-8218-2076-1},
	MRCLASS = {01A75},
	MRNUMBER = {2280876},
}

@article {Schweitzer:AShortProofThatMnAIsLocal,
	AUTHOR = {Schweitzer, L. B.},
	TITLE = {A short proof that {$M_n(A)$} is local if {$A$} is local and
	{F}r\'echet},
	JOURNAL = {Internat. J. Math.},
	FJOURNAL = {International Journal of Mathematics},
	VOLUME = {3},
	YEAR = {1992},
	NUMBER = {4},
	PAGES = {581--589},
	ISSN = {0129-167X,1793-6519},
	MRCLASS = {46H05 (46J05)},
	MRNUMBER = {1168361},
	MRREVIEWER = {Gustavo\ Corach},
}

@book {Connes:NoncommutativeGeometry,
	AUTHOR = {Connes, A.},
	TITLE = {Noncommutative geometry},
	PUBLISHER = {Academic Press, Inc., San Diego, CA},
	YEAR = {1994},
	PAGES = {xiv+661},
	ISBN = {0-12-185860-X},
	MRCLASS = {46Lxx (19K56 22D25 58B30 58G12 81T13 81V22 81V70)},
	MRNUMBER = {1303779},
	MRREVIEWER = {John\ Roe},
}

@misc{BrixHumeLi:MinimalCovers,
	title={Minimal covers with continuity-preserving transfer operators for topological dynamical systems}, 
	author={K. A. Brix and J. B. Hume and X. Li},
	year={2024},
	eprint={2408.11917},
	archivePrefix={arXiv},
	primaryClass={math.DS},
	url={https://arxiv.org/abs/2408.11917}, 
}

@article {Grigorchuk:DegreesOfGrowth,
	AUTHOR = {Grigorchuk, R. I.},
	TITLE = {Degrees of growth of finitely generated groups and the theory
	of invariant means},
	JOURNAL = {Izv. Akad. Nauk SSSR Ser. Mat.},
	FJOURNAL = {Izvestiya Akademii Nauk SSSR. Seriya Matematicheskaya},
	VOLUME = {48},
	YEAR = {1984},
	NUMBER = {5},
	PAGES = {939--985},
	ISSN = {0373-2436},
	MRCLASS = {20F05 (43A07)},
	MRNUMBER = {764305},
	MRREVIEWER = {P.\ Gerl},
}

@book {Renault:AGroupoidApproach,
	AUTHOR = {Renault, J.},
	TITLE = {A groupoid approach to {$C\sp{\ast} $}-algebras},
	SERIES = {Lecture Notes in Mathematics},
	VOLUME = {793},
	PUBLISHER = {Springer, Berlin},
	YEAR = {1980},
	PAGES = {ii+160},
	ISBN = {3-540-09977-8},
	MRCLASS = {46Lxx (22D25 22D40)},
	MRNUMBER = {584266},
	MRREVIEWER = {A.\ K.\ Seda},
}

@article {Gardella:AModernLook,
	AUTHOR = {Gardella, E.},
	TITLE = {A modern look at algebras of operators on {$L^p$}-spaces},
	JOURNAL = {Expo. Math.},
	FJOURNAL = {Expositiones Mathematicae},
	VOLUME = {39},
	YEAR = {2021},
	NUMBER = {3},
	PAGES = {420--453},
	ISSN = {0723-0869,1878-0792},
	MRCLASS = {43A15 (22D20 43A65 47L10)},
	MRNUMBER = {4314026},
	MRREVIEWER = {Roc\'io\ D\'iaz Mart\'in},
}

@misc{MaWuAlmostElementarinessAndFiberwiseAmenabilityForEtaleGroupoids,
	title={Almost elementariness and fiberwise amenability for \'etale groupoids}, 
	author={X. Ma and J. Wu},
	year={2020},
	eprint={2011.01182},
	archivePrefix={arXiv},
	primaryClass={math.OA},
	url={https://arxiv.org/abs/2011.01182}, 
}

@article {Paravicini:TheBostConjectureAndProperBanachAlgebras,
	AUTHOR = {Paravicini, W.},
	TITLE = {The {B}ost conjecture and proper {B}anach algebras},
	JOURNAL = {J. Noncommut. Geom.},
	FJOURNAL = {Journal of Noncommutative Geometry},
	VOLUME = {7},
	YEAR = {2013},
	NUMBER = {1},
	PAGES = {191--202},
	ISSN = {1661-6952,1661-6960},
	MRCLASS = {43A20 (19K35 22A22 22D15)},
	MRNUMBER = {3032815},
}

@misc{ElkiearPooya:PropertyTForBanachAlgebras,
	title={Property (T) for Banach algebras}, 
	author={E. M. Elkiær and S. Pooya},
	year={2023},
	eprint={2310.18136},
	archivePrefix={arXiv},
	primaryClass={math.FA},
	url={https://arxiv.org/abs/2310.18136}, 
}

@article {SameiWiersma:ExoticCstarAlgebrasOfGeometricGroups,
	AUTHOR = {Samei, E. and Wiersma, M.},
	TITLE = {Exotic {$\rm C^*$}-algebras of geometric groups},
	JOURNAL = {J. Funct. Anal.},
	FJOURNAL = {Journal of Functional Analysis},
	VOLUME = {286},
	YEAR = {2024},
	NUMBER = {2},
	PAGES = {Paper No. 110228, 32},
	ISSN = {0022-1236,1096-0783},
	MRCLASS = {46L05 (22D25 43A65)},
	MRNUMBER = {4664988},
	MRREVIEWER = {Xiao\ Chen},
}

@article {SameiWiersma:QuasiHermitianLocallyCompactGroupsAreAmenable,
	AUTHOR = {Samei, E. and Wiersma, M.},
	TITLE = {Quasi-{H}ermitian locally compact groups are amenable},
	JOURNAL = {Adv. Math.},
	FJOURNAL = {Advances in Mathematics},
	VOLUME = {359},
	YEAR = {2020},
	PAGES = {106897, 25},
	ISSN = {0001-8708,1090-2082},
	MRCLASS = {43A20 (22D25 43A07 43A15)},
	MRNUMBER = {4031113},
	MRREVIEWER = {Pekka\ Salmi},
}

@article{AustadOrtega:GroupoidsAndHermitianBanachstarAlgebras,
	AUTHOR = {Austad, A. and Ortega, E.},
	TITLE = {Groupoids and {H}ermitian {B}anach {$^*$}-algebras},
	JOURNAL = {Internat. J. Math.},
	FJOURNAL = {International Journal of Mathematics},
	VOLUME = {33},
	YEAR = {2022},
	NUMBER = {14},
	PAGES = {Paper No. 2250090, 25},
	ISSN = {0129-167X,1793-6519},
	MRCLASS = {22D25 (22A22 43A15 46H05 47B01)},
	MRNUMBER = {4536260},
	MRREVIEWER = {Judith\ A.\ Packer},
}

@article {ChoiGardellaThiel:RigidityResultsForLpOperatorAlgebrasAndApplications,
    AUTHOR = {Choi, Y. and Gardella, E. and Thiel, H.},
     TITLE = {Rigidity results for {$L^p$}-operator algebras and
              applications},
   JOURNAL = {Adv. Math.},
  FJOURNAL = {Advances in Mathematics},
    VOLUME = {452},
      YEAR = {2024},
     PAGES = {Paper No. 109747, 47},
      ISSN = {0001-8708,1090-2082},
   MRCLASS = {43A15 (22A22 43A65 47L10)},
  MRNUMBER = {4767397},
}

@article {GardellaThiel:RigidityResultsGroupsForLpOperatorAlgebras,
	AUTHOR = {Gardella, E. and Thiel, H.},
	TITLE = {Isomorphisms of algebras of convolution operators},
	JOURNAL = {Ann. Sci.Éc. Norm. Supér.},
	FJOURNAL = {ANNALES SCIENTIFIQUES DE L'ÉCOLE NORMALE SUPÉRIEURE},
	VOLUME = {55 (5)},
	YEAR = {2022},
	PAGES = {1433–1471},
	ISSN = { },
	MRCLASS = {43A15 (22A22 43A65 47L10)},
	MRNUMBER = {4767397},
}

@article{AustadOrtegaPalmstrom:PolynomialGrowthAndPropertyRDpForEtaleGroupoids,
	title={Polynomial growth and property $ RD_p $ for {\'e}tale groupoids with applications to $ K $-theory},
	author={Austad, A. and Ortega, E. and Palmstr{\o}m, M.},
	journal={arXiv preprint arXiv:2304.14458},
	year={2023}
}

@book{LindMarcus:SymbolicDynamics,
	AUTHOR = {Lind, D. and Marcus, B.},
	TITLE = {An Introduction to Symbolic Dynamics and Coding},
	SERIES = {},
	VOLUME = {5},
	EDITION = {Second},
	PUBLISHER = {Cambridge University Press, Cambridge},
	YEAR = {1995},
	PAGES = {xx+300},
	ISBN = {0-521-55124-2},
	MRCLASS = {46L80 (19Kxx 58G12)},
	MRNUMBER = {1656031},
}

@article{LauterMonthubertNistor:SpectralInvarianceForCertainAlgebrasOfPseudodifferentialOperators,
	AUTHOR = {Lauter, R. and Monthubert, B. and Nistor, V.},
	TITLE = {Spectral invariance for certain algebras of pseudodifferential
	operators},
	JOURNAL = {J. Inst. Math. Jussieu},
	FJOURNAL = {Journal of the Institute of Mathematics of Jussieu. JIMJ.
	Journal de l'Institut de Math\'ematiques de Jussieu},
	VOLUME = {4},
	YEAR = {2005},
	NUMBER = {3},
	PAGES = {405--442},
	ISSN = {1474-7480,1475-3030},
	MRCLASS = {47G30 (35J15 35S05 46L87 47L80 58J40)},
	MRNUMBER = {2197064},
}

@article{OztopSameiShepelska:TwistedOrliczAlgebrasAndCompleteIsomorphismToOperatorAlgebras,
	AUTHOR = {\"Oztop, S. and Samei, E. and Shepelska, V.},
	TITLE = {Twisted {O}rlicz algebras and complete isomorphism to operator
	algebras},
	JOURNAL = {J. Math. Anal. Appl.},
	FJOURNAL = {Journal of Mathematical Analysis and Applications},
	VOLUME = {477},
	YEAR = {2019},
	NUMBER = {2},
	PAGES = {1114--1132},
	ISSN = {0022-247X,1096-0813},
	MRCLASS = {43A15 (46E30 46L10 47L15 47L25)},
	MRNUMBER = {3955013},
	MRREVIEWER = {Volker\ Runde},
}

@article{Samei&Shepelska:NormcontrolledInverseionInWeightedConvolutionAlgebras,
	AUTHOR = {Samei, E. and Shepelska, V.},
	TITLE = {Norm-controlled inversion in weighted convolution algebras},
	JOURNAL = {J. Fourier Anal. Appl.},
	FJOURNAL = {The Journal of Fourier Analysis and Applications},
	VOLUME = {25},
	YEAR = {2019},
	NUMBER = {6},
	PAGES = {3018--3044},
	ISSN = {1069-5869,1531-5851},
	MRCLASS = {43A10 (43A15 46L05)},
	MRNUMBER = {4029169},
	MRREVIEWER = {Mohammad\ Fozouni},
}

@article {BondarenkoEtAl2012,
    AUTHOR = {Bondarenko, I. and Ceccherini-Silberstein, T. and
              Donno, A. and Nekrashevych, V.},
     TITLE = {On a family of {S}chreier graphs of intermediate growth
              associated with a self-similar group},
   JOURNAL = {European J. Combin.},
  FJOURNAL = {European Journal of Combinatorics},
    VOLUME = {33},
      YEAR = {2012},
    NUMBER = {7},
     PAGES = {1408--1421},
      ISSN = {0195-6698,1095-9971},
   MRCLASS = {05C25},
  MRNUMBER = {2923458},
}

@misc{KontogeorgiouWinter2022,
Author = {Kontogeorgiou, G. and Winter, M.},
Title = {(Random) Trees of Intermediate Uniform Growth},
Year = {2022},
Eprint = {arXiv:2212.01883},
}

@article {Lehner2016,
    AUTHOR = {Lehner, F.},
     TITLE = {Distinguishing graphs with intermediate growth},
   JOURNAL = {Combinatorica},
  FJOURNAL = {Combinatorica. An International Journal on Combinatorics and
              the Theory of Computing},
    VOLUME = {36},
      YEAR = {2016},
    NUMBER = {3},
     PAGES = {333--347},
      ISSN = {0209-9683,1439-6912},
   MRCLASS = {05C25 (05C15 20B27)},
  MRNUMBER = {3521117},
MRREVIEWER = {Simon\ M.\ Smith},
}

@article {MiasnikovSavchuk2015,
    AUTHOR = {Miasnikov, A. and Savchuk, D.},
     TITLE = {An example of an automatic graph of intermediate growth},
   JOURNAL = {Ann. Pure Appl. Logic},
  FJOURNAL = {Annals of Pure and Applied Logic},
    VOLUME = {166},
      YEAR = {2015},
    NUMBER = {10},
     PAGES = {1037--1048},
      ISSN = {0168-0072,1873-2461},
   MRCLASS = {68Q45 (03D05 20F10 20F65)},
  MRNUMBER = {3356618},
}

@misc{PhillipsSimplicity2019,
Author = {N. C. Phillips},
Title = {Simplicity of reduced group Banach algebras},
Year = {2019},
Eprint = {arXiv:1909.11278},
}

@article {Choi15directlyfinite,
    AUTHOR = {Choi, Y.},
     TITLE = {Directly finite algebras of pseudofunctions on locally compact
              groups},
   JOURNAL = {Glasg. Math. J.},
  FJOURNAL = {Glasgow Mathematical Journal},
    VOLUME = {57},
      YEAR = {2015},
    NUMBER = {3},
     PAGES = {693--707},
      ISSN = {0017-0895,1469-509X},
   MRCLASS = {22D15 (22D25 43A15 46L05)},
  MRNUMBER = {3395342},
MRREVIEWER = {Massoud\ Amini},
}

@article {ChungLi2018roerigidity,
    AUTHOR = {Chung, Y. C. and Li, K.},
     TITLE = {Rigidity of {$\ell^p$} {R}oe-type algebras},
   JOURNAL = {Bull. Lond. Math. Soc.},
  FJOURNAL = {Bulletin of the London Mathematical Society},
    VOLUME = {50},
      YEAR = {2018},
    NUMBER = {6},
     PAGES = {1056--1070},
      ISSN = {0024-6093,1469-2120},
   MRCLASS = {46L85 (46H15 51K05)},
  MRNUMBER = {3891943},
MRREVIEWER = {Ian\ Charlesworth},
}

@article {CortinasRodriguez19OrientedGraphs,
    AUTHOR = {Corti\~nas, G. and Rodr\'iguez, M. E.},
     TITLE = {{$L^p$}-operator algebras associated with oriented graphs},
   JOURNAL = {J. Operator Theory},
  FJOURNAL = {Journal of Operator Theory},
    VOLUME = {81},
      YEAR = {2019},
    NUMBER = {1},
     PAGES = {225--254},
      ISSN = {0379-4024,1841-7744},
   MRCLASS = {47L10 (46L55)},
  MRNUMBER = {3920695},
MRREVIEWER = {Flavius\ Pater},
}

@article {PhillipsViola2020,
    AUTHOR = {Phillips, N. C. and Viola, M. G.},
     TITLE = {Classification of spatial {$L^p$} {AF} algebras},
   JOURNAL = {Internat. J. Math.},
  FJOURNAL = {International Journal of Mathematics},
    VOLUME = {31},
      YEAR = {2020},
    NUMBER = {13},
     PAGES = {2050088, 41},
      ISSN = {0129-167X,1793-6519},
   MRCLASS = {47L10 (46L35)},
  MRNUMBER = {4192445},
MRREVIEWER = {Eusebio\ Gardella},
}

\end{document}